\documentclass[10pt]{amsart}
\usepackage{amsmath, amscd, amsthm} 
\usepackage{amssymb, amsfonts} 
\usepackage{enumerate}
\usepackage[all]{xy} 
\usepackage[margin=1.5in]{geometry}
\usepackage{booktabs}
\usepackage{aliascnt}
\usepackage[svgnames]{xcolor} 
\usepackage{hyperref}
 \definecolor{dark-red}{rgb}{0.4,0.15,0.15}
\hypersetup{
    colorlinks, linkcolor=dark-red,
    citecolor=DarkBlue, urlcolor=MediumBlue
}
\newcommand{\Q}{\mathbb{Q}} 
\newcommand{\QQ}{\Q}
\newcommand{\C}{\mathbb{C}} 
\newcommand{\A}{\mathbb{A}}
\renewcommand{\P}{\mathbb{P}}
\newcommand{\Z}{\mathbb{Z}}
\newcommand{\ZZ}{\mathbb{Z}}
\newcommand{\R}{\mathbb{R}}
\newcommand{\RR}{\R}
\renewcommand{\L}{\mathbb{L}}

\newcommand{\iso}{\cong}
\newcommand{\wkeq}{\simeq}

\newcommand{\Sch}{\mathrm{Sch}}
\newcommand{\Sm}{\mathrm{Sm}}
\newcommand{\sch}{\mathrm{Sch}}
\newcommand{\sm}{\mathrm{Sm}}

\newcommand{\tr}{\mathrm{tr}}
\newcommand{\RRe}{\mathrm{Re}}

\newcommand{\ssetb}{\mathrm{sSet}_{\bullet}}
\newcommand{\btop}{\mathrm{Top}_{\bullet}}

\newcommand{\spre}{\mathrm{sPre}}

\newcommand{\Sp}{\mathrm{Sp}}
\newcommand{\sspt}{\mathrm{Spt}^{\Sigma}}
\newcommand{\MS}{\mathrm{Spc}_{\bullet}}

\newcommand{\SH}{\mathrm{SH}}
\newcommand{\HH}{\mathrm{H}}
\newcommand{\MM}{\mathrm{M}}
\newcommand{\Or}{\mathrm{Or}}

\renewcommand{\SS}{\mathbb{S}}

\newcommand{\id}{\mathrm{id}}

\newcommand{\End}{\mathrm{End}}

\newcommand{\res}{\mathrm{res}}
\newcommand{\ind}{\mathrm{ind}}

\newcommand{\mcal}[1]{\mathcal{#1}}
\newcommand{\ul}[1]{\underline{\smash{#1}}}

\newcommand{\et}{\acute{e}t}

\renewcommand{\setminus}{\smallsetminus}

\DeclareMathOperator*{\colim}{\mathrm{colim}}

\DeclareMathOperator*{\holim}{\mathrm{holim}}

\DeclareMathOperator{\cofiber}{\mathrm{cofiber}}

\DeclareMathOperator{\spec}{\mathrm{Spec}}

\DeclareMathOperator{\Hom}{Hom}

\DeclareMathOperator{\rk}{rk}

\DeclareMathOperator{\Sym}{Sym}

\DeclareMathOperator{\sing}{\mathrm{Sing}}

\DeclareMathOperator{\Gal}{\mathrm{Gal}}

\numberwithin{equation}{section} 

\theoremstyle{plain}

\newaliascnt{theorem}{equation}  
\newtheorem{theorem}[theorem]{Theorem}  
\aliascntresetthe{theorem}

\newaliascnt{proposition}{equation}  
\newtheorem{proposition}[proposition]{Proposition}
\aliascntresetthe{proposition}

\newaliascnt{lemma}{equation}  
\newtheorem{lemma}[lemma]{Lemma}
\aliascntresetthe{lemma}

\newaliascnt{corollary}{equation}  
\newtheorem{corollary}[corollary]{Corollary}
\aliascntresetthe{corollary}

\newaliascnt{claim}{equation}  

\aliascntresetthe{claim}

\newaliascnt{conjecture}{equation}  
\newtheorem{conjecture}[conjecture]{Conjecture}
\aliascntresetthe{conjecture}

 \theoremstyle{definition}

\newaliascnt{definition}{equation}  
\newtheorem{definition}[definition]{Definition}
\aliascntresetthe{definition}

\newaliascnt{example}{equation}  

\aliascntresetthe{example}

\newaliascnt{remark}{equation}  
\newtheorem{remark}[remark]{Remark}
\aliascntresetthe{remark}

\newcommand{\aref}[1]{\autoref{#1}}

\begin{document}
\title[Galois Equivariance and Stable Motivic Homotopy Theory]{Galois Equivariance and \\Stable Motivic Homotopy Theory}
\author{J. Heller}
\address{University of Illinois Urbana-Champaign}
\email{jeremiahheller.math@gmail.com}
\author{K. Ormsby}
\address{Reed College}
\email{ormsbyk@reed.edu}

\begin{abstract}
For a finite Galois extension of fields $L/k$ with Galois group $G$,
we study a functor from the $G$-equivariant stable homotopy category
to the stable motivic homotopy category over $k$ induced by the
classical Galois correspondence.  We show that after completing at a prime and $\eta$ (the motivic Hopf map) this results in a full and
faithful embedding whenever $k$ is real closed and $L=k[i]$. 
It is a full and faithful embedding after $\eta$-completion if a motivic version of Serre's finiteness theorem is valid.
We
produce strong necessary conditions on the field extension $L/k$ for this functor to
be full and faithful.  
Along the way, we produce several results on the stable $C_2$-equivariant Betti realization functor and prove convergence theorems for the $p$-primary $C_2$-equivariant Adams spectral sequence.
\end{abstract}

\maketitle

\section{Introduction}
The stable versions of equivariant and motivic homotopy theory play 
important roles in the geometry of manifolds, algebraic
cycles, and quadratic forms.  Stable equivariant
homotopy theory is the study of topological spaces equipped with a
group action up to stable
equivariant weak equivalence.  It has recently found stunning application \cite{HHR}
to the Kervaire problem, playing an essential role in the proof that there are no smooth framed
manifolds of Kervaire invariant one in dimensions greater than
$126$.  
Via the work of Devinatz and Hopkins \cite{DevHop}, stable
equivariant homotopy theory 
controls the chromatic decomposition of stable homotopy theory. It is also essential to the
study of topological Hochschild homology \cite{BHM:cyclo}.

Motivic homotopy theory is a homotopy theory of schemes in
which the affine line plays the role of the unit interval.  Its study
was initiated by Morel and Voevodsky \cite{MV:A1} in work related to
Rost and Voevodsky's resolution of the Bloch-Kato conjectures on Milnor
$K$-theory and Galois cohomology \cite{Voevodsky:Z2,Voevodsky:Zl}.  Its stable version plays an
essential role in the theory of motives and motivic cohomology
\cite{OrangeBook}.  This
circle of ideas led to the resolution of the Milnor conjecture on
quadratic forms \cite{OVV} and the Quillen-Lichtenbaum conjecture, a powerful result linking 
algebraic $K$-theory and values of Dedekind $\zeta$-functions via a
``homotopy limit problem'' phrased in the language of stable
equivariant homotopy \cite{DF:algetK}.  Stable motivic
homotopy theory also opens new vistas, such as the study of algebraic
cobordism \cite{V:ICM}.

\emph{The purpose of this paper is to study how equivariant and
  motivic stable homotopy theory are
  related via the classical Galois correspondence.}

 \vspace{.4cm}
 
 A fundamental computation in stable motivic homotopy theory 
 is the identification of the endomorphism ring of the motivic sphere spectrum by Morel \cite{Morel:sphere}. In \emph{loc.~cit.}~Morel shows that  
  $\End_{\SH_{k}}(\SS_{k})$  is isomorphic to the Grothendieck-Witt
  group $GW(k)$ of nondegenerate  quadratic forms over a perfect field $k$.  
  It is now a classical fact, going back to Segal and tom Dieck,  that the endomorphism ring $\End_{\SH_{G}}(\SS_{G})$ of the equivariant sphere spectrum in the equivariant stable homotopy category is equal to the Burnside ring $A(G)$
  of finite $G$-sets.

When $L/k$ is a finite Galois extension with Galois group $G$, Dress
\cite[Appendix B]{Dress} (see also, \cite[\S 4]{BeaulieuPalfrey}) constructs a ring homomorphism 
$A(G)\to GW(k)$ relating these two fundamental invariants. In fact, the Galois correspondence can be stabilized to yield
 a strong symmetric monoidal triangulated functor from the stable
 $G$-equivariant homotopy category to the stable motivic homotopy
 category over $k$,
 $$
 c_{L/k}^*:\SH_{G}\to \SH_{k}.
 $$ 
This relies on work of P. Hu \cite{Hu:base}. 
 When $L=k$,  
 $c_{L/L}^*$ is simply the functor  induced by sending a simplicial
 set to its associated constant motivic space.  When $L=k$ is
 algebraically closed of characteristic zero, Levine
 \cite{Levine:comparison} has recently shown that $c^*_{L/L}$ is a
 full and faithful embedding, but this is not the general case for
 $c_{L/k}^*$.  Indeed,  
 the Burnside ring $A(G)$ is always torsion free while $GW(k)$ can  in
 general contain torsion, which eliminates the possibility of
 $c_{L/k}^*$ inducing an isomorphism $A(G)\cong GW(k)$. 
 However, if $k$ is a real closed
 field then $GW(k)$ and $A(C_{2})$ are isomorphic so one might still hope that Levine's embedding theorem can be generalized to real closed fields. Our main result, proved in \aref{thm:textmain} and \aref{thm:textmain2} below, is that this indeed is the case after $(p,\eta)$-completion.  Here $p$ is a prime and $\eta$ is the motivic Hopf map induced by the canonical projection $\AA^2\smallsetminus 0\to \mathbb{P}^1$.  (Details on $(p,\eta)$-completion are provided at the start of \aref{sec:main}.)  Moreover, the functor is a full and faithful embedding after $\eta$-completion alone if 
 $\pi_{n}(\SS_{k})_{\Q} =0$ for $n>0$.  The vanishing of these higher homotopy groups would be a motivic version of the classical result of Serre on the homotopy groups of spheres and is already known to be true when $-1$ is a sum of squares in the basefield.

\begin{theorem}\label{thm:main}
Let $k$ be a real closed field and $L=k[i]$ be its algebraic closure. Then for any prime $p$ the functor
$$
c_{L/k}^*:\SH_{C_2}\to \SH_k
$$
is a full and faithful embedding after $(p,\eta)$-completion. If $\pi_{n}(\SS_{k})_{\Q} = 0$ for $n>0$ it is a full and faithful embedding after $\eta$-completion.
\end{theorem}

It is a consequence of \cite[Theorem 1]{HKO} that $(2,\eta)$-completion is the same as $2$-completion when $k$ is real closed, so the above theorem specializes at $p=2$ to say that $c_{L/k}^*$ is full and faithful after $2$-completion when $k$ is real closed.

\begin{remark}
In order to deduce integral full faithfulness of $c_{L/k}^*$ from $\eta$-complete full faithfulness, one would need to control the $\eta$-periodic (\emph{i.e.}, $\eta$-inverted) stable homotopy categories as well.  Recent work of Guillou-Isaksen and Andrews studies the $\eta$-periodic $2$-complete sphere over $\mathbb{C}$ from a computational perspective, but there aren't many techniques developed for working with purely $\eta$-periodic objects in general.
\end{remark}

\subsection{Computational ramifications}
 Our embedding result
has significant implications for (Picard-graded) stable homotopy groups of
spheres in the $C_2$-equivariant and real closed motivic settings.
Recall that the representation spheres $S^{m+n\sigma}$ are invertible
in $\SH_{C_2}$ where $S^{m+n\sigma}$ is the one-point compactification
of $m$ copies of the one-dimensional real trivial representation and
$n$ copies of the real sign representation.  As such $\ZZ\oplus
\ZZ\{\sigma\}$ is a subgroup of the Picard group of invertible objects
in $\SH_{C_2}$, and it is common to consider the bigraded stable
homotopy groups $\pi_{m+n\sigma}X = [S^{m+n\sigma},X]_{C_2}$ of a
$C_2$-spectrum $X$.  When $k$ is real closed and $L=k[i]$,
\aref{thm:main} implies that $c_{L/k}^*:\pi_{m+n\sigma} (\SS_{C_2})^{\wedge}_{p,\eta}
\cong
[c_{L/k}^* S^{m+n\sigma},(\SS_k)^{\wedge}_{p,\eta}]_k$.  We will
see that $c_{L/k}^*S^{m+n\sigma}\simeq S^m \wedge (S^L)^{\wedge n}$ where
$S^L$ is the unreduced suspension of $\spec(L)$. 
By a
theorem of P. Hu \cite{Hu:base}, $S^L$ is invertible and $\ZZ\oplus
\ZZ\{L\}$ is a subgroup of the Picard group of $\SH_k$.  We emphasize
that $S^L$ is \emph{not} weakly equivalent to $\A^1\smallsetminus
\{0\}$ and  this is \emph{not} the ``standard'' bigrading in motivic
homotopy theory. 

Regardless, if we set $S^{m+nL} = S^m\wedge (S^L)^{\wedge n}$ and make
the natural definition of $\pi_{m+nL}$, we see that $c_{L/k}^*$ induces
isomorphisms
\[
  \pi_{m+n\sigma} (\SS_{C_2})^{\wedge}_{p,\eta} \xrightarrow{\iso} 
  \pi_{m+nL}(\SS_k)_{p,\eta}^{\wedge}
\]
for all $m,n\in \ZZ$ under the conditions of  \aref{thm:main}.
It is an observation of D. Dugger that the same result does not hold
if $S^L$ is replaced by $\A^1\smallsetminus \{0\}$.

The $C_2$-equivariant stable stems were studied by Araki
and Iriye via Toda-style methods.  In \cite{AI:involutions}, they
compute the groups $\pi_{m+n\sigma}\SS_{C_2}$ for $m+n\le 8$.  In
particular, they compute the groups $\pi_m \SS_{C_2}$ for $m\le 8$, so
\aref{thm:main} implies the following
corollary. 

\begin{corollary} \label{cor:comp}
If $k$ is a real closed field, then 
$\pi_m (\SS_k)^{\wedge}_{2}$, $0\le m\le 8$, is the $2$-completion of the values
displayed in the following table.

\begin{center}
\renewcommand{\arraystretch}{1.2}
\begin{tabular}{@{}l c c c c c c c c c @{}}
\toprule
$ m$ &$0$&$1$&$2$&$3$&$4$&$5$&$6$&$7$&$8$\\ \midrule
$ \pi_m\SS_k$
&$\ZZ^2$ 
&$(\ZZ/2)^3$
&$(\ZZ/2)^3$
&$\begin{aligned}( \ZZ & /24)^2 \\ \oplus & \ZZ/8 \end{aligned}$
&$\ZZ/2$
&$0$
&$(\ZZ/2)^3$
&{$\begin{aligned}[c]& (\ZZ/240)^2\oplus \\& \ZZ/16\oplus \ZZ/2\end{aligned}$}
&$(\ZZ/2)^7$ \\
\bottomrule
\end{tabular}
\end{center}
\end{corollary}

In addition, in \aref{cor:Morel} we show that the 2-complete version of Morel's conjecture on $\pi_1(\SS_k)$  holds for real closed fields.  The integral version of this conjecture says that, for a general base field $F$, there is a
short exact sequence
\[
  0\to K^M_2(F)/24\to \pi_1\SS_F\to F^\times/(F^\times)^2\oplus
  \ZZ/2\to 0.
\]
The second-named author and P.~{\O}stv{\ae}r have previously verified the integral version of 
Morel's conjecture for fields of cohomological dimension less than
three \cite{OrmsbyOstvaer:pi1}.

While these immediate applications transfer information from
$C_2$-equivariant to motivic homotopy over a real closed field,
future work should leverage motivic homotopy to produce
$C_2$-equivariant computations.  In particular, the dual motivic Steenrod
algebra is smaller than its equivariant counterpart, making Adams and
Adams-Novikov spectral sequence computations more approachable.  The
authors plan to apply these tools over the field 
$\RR$ of real numbers (with the above exotic Picard grading) in order
to extend our computational understanding of the stable
$C_2$-equivariant homotopy category.

\subsection{Galois correspondence and motivic homotopy theory}
An intriguing viewpoint on our embedding theorem is as a generalization of the classical Galois correspondence in the case of real closed fields. Indeed, if $L/k$ is a finite Galois extension with Galois group $G$ then 
the Galois correspondence is an equivalence between the category of finite $G$-sets and the category of finite \'etale $k$-algebras.
Restricting to the orbit category, this correspondence gives the functor
$$
c_{L/k}:\Or_{G}\to \Sm/k
$$ 
to smooth $k$-schemes which is explicitly given on objects by $c_{L/k}(G/H) = \spec(L^{H})$. 
As recorded in \aref{thm:construction}, this functor can be stabilized, yielding a strong symmetric monoidal, triangulated functor
$$
c_{L/k}^*:\SH_{G}\to \SH_{k}.
$$
It is not hard to see that the unstable version of this functor
induces a full and faithful embedding from the unstable
$G$-equivariant homotopy category to the unstable motivic homotopy
category over $k$ (see \aref{lemma:uff}). Note, though, that the
stable equivariant homotopy category is formed by stabilizing with
respect to representation spheres while the motivic homotopy category
is formed by stabilizing with respect to $\P^1$.  Hence there is no
reason for this pleasant relationship between the two categories to
remain after stabilization, yet it does in special cases. In fact, we
can say something slightly more precise.  The image of $c_{L/k}^*$ is always contained in the subcategory $\mcal{E}_{k}$ of $\SH_{k}$ which is generated by the finite \'etale $k$-algebras.\footnote{Here ``generated'' means that $\mcal{E}_{k}$ is the smallest localizing subcategory of $\SH_{k}$ containing all (suspension spectra of) finite \'etale $k$-algebras. } Our result can thus be rephrased as an equivalence of triangulated categories between $\SH_{C_{2}}$ and $\mcal{E}_{k}$ when $k$ is real closed.

This translation of stable motivic homotopy over $k$ into stable
$G$-equivariant homotopy for $G = \Gal(L/k)$ will not work for general
finite Galois extensions $L/k$.  Indeed, in 
\aref{thm:necessary} we show that $c^*_{L/k}$ induces an isomorphism $A(G)\to GW(k)$ if and only if
either $k$ is quadratically closed and $L=k$, or $k$ is euclidean\footnote{A field $k$ is euclidean if $-1$ is not a sum of squares in $k$ and $[k^\times \smash : \, (k^\times)^2]=2$.}
and $L=k[i]$. This implies in particular that $c^*_{L/k}$ cannot be full and faithful if $L/k$ is not of this special form.

\subsection{Outline of the proof}
Our main theorem is directly inspired by M.~Levine's theorem on full
faithfulness of the constant presheaf functor \cite{Levine:comparison}, and our methods
are, largely, in the same spirit as his.  That said, Levine's
arguments 
rely on the convergence of the slice spectral sequence, a result not
yet known over fields with infinite cohomological dimension. To remedy
this situation, we compare the motivic and equivariant Adams spectral
sequences.

Let $k$ be real closed and set $L=k[i]$ so that $G=C_2$ is cyclic of
order $2$.  By a density argument, to show that  $c_{L/k}^*$ is full and faithful after $\eta$-completion, it suffices to show that  $c_{L/k}^*$ induces isomorphisms
\[
  [S^{n}\wedge X,Y]_{C_2}\xrightarrow{\iso} [S^{n}\wedge c_{L/k}^*(X)^{\wedge}_{\eta},c_{L/k}^*(Y)^{\wedge}_{\eta}]_k
\] 
where $X,Y$ take values in the set $\{(\SS_{C_2})^{\wedge}_{\eta}, C_{2\,+}\wedge (\SS_{C_2})^{\wedge}_{\eta}\}$. The key case is when $k$ admits a real embedding and in this case we can use the $C_{2}$-equivariant Betti realization. The computation is broken up into pieces: the
$(p,\eta)$-completed sphere (for any prime $p$) and the rationalized $\eta$-complete sphere. The computation concerning the latter object relies on the conjectural motivic version of Serre's finiteness theorem and so the $\eta$-complete version of the embedding theorem is conditional upon the validity of this conjecture. Of course, the full and faithful embedding of $(p,\eta)$-completed homotopy categories holds independent of this conjecture. 
In the $(p,\eta)$-complete case, we identify the $C_2$-equivariant Betti realization of the motivic
Adams spectral sequence with the $C_2$-equivariant Adams spectral
sequence based on the Bredon cohomology spectrum
$\HH\ul{\ZZ/p}$.  
We establish an equivariant version of Suslin-Voevodsky's theorem on Suslin homology which implies that  the realization
induces an isomorphism on weight zero components of the $E_1$-pages from which we deduce the result in this case.

\subsection{Comments on realization and profinite Galois extensions}
We conclude by making a few comments on the role of
 ``realization'' functors.   M.~Levine uses the Betti realization
functor $\RRe_B:\SH_L\to \SH$ for algebraically closed subfields $L$ of $\C$ to prove his full faithfulness theorem in
\cite{Levine:comparison}.  Since $\RRe_B\circ c^* = \id$, the constant presheaf functor is always faithful for any $k\subseteq \C$.  Levine's
innovation was  to compare the Betti realization of the slice spectral sequence for
the motivic sphere spectrum over an algebraically closed field with
the Novikov spectral sequence in topology.  An isomorphism between the $E_{2}$-terms of 
these spectral sequences implies an isomorphism on stable homotopy
groups of spheres which ultimately implies the fullness result.

When $k$ has a real embedding, then there is an associated
$C_2$-equivariant Betti realization $\RRe_B^{C_2}:\SH_k\to \SH_{C_2}$. 
As previously mentioned, we cannot use the slice spectral sequence to prove our embedding theorem, but our arguments still rely on using (equivariant) Betti realization to compare some spectral sequences (namely the motivic and equivariant Adams spectral sequences). Again, faithfulness of $c^*_{k[i]/k}$ is easy because $\RRe_B^{C_2}\circ c_{k[i]/k}^* = \id$.

 Suppose $L =
\bar{k}$ is the algebraic closure of $k$ and $G$ is the absolute Galois group $\Gal(L/k)$, which is a profinite group.
A natural question
is whether the main theorem of this paper extends to a full
faithfulness theorem for $G$-equivariant stable homotopy inside of
$\SH_k$.  In order to precisely state such a question, though, one
would need an appropriate notion of genuine $G$-spectra and $G$-stable
homotopy when $G$ is profinite.  Proposals for this category are
contained in \cite{Fausk:pro,Quick:pro}, and C.~Barwick has communicated
ideas on an alternate formulation to the authors.  Whichever model is
chosen, one would hope that it would admit well-behaved functors
\[
  c_{L/k}^*:\SH_G\to \SH_k\quad\text{and}\quad \RRe_B^G:\SH_k\to \SH_G
\]
such that $\RRe_{B}^G\circ c_{L/k}^*$ is some form of pro-completion
of the identity functor.    This would result in a pro-faithfulness
theorem, at which point one could examine fullness properties as well.
The authors hope to pursue this line of inquiry
in future research.

\subsection{Organization of the paper}
We prove our main theorem in \S\ref{sec:main} according to the
strategy outlined above.  We then deduce several
interesting corollaries, 
including our Picard-graded homotopy
comparison (\aref{cor:groups}), 
Morel's conjecture on $\pi_1\SS_k$ for
real closed fields (\aref{cor:Morel}), and a relative version of our
theorem comparing full faithfulness of $c_{L/L}^*$ and full
faithfulness of $c_{L/k}^*$ (\aref{cor:rel}).

In \S\ref{sec:classical}, we study the effect of
$c_{L/k}^*$ on the endomorphism ring of the sphere spectrum.  
We show that it induces an isomorphism if and only if either $k$ is quadratically closed and
$L=k$, or $k$ is euclidean and $L=k[i]$ (\aref{thm:necessary}); 
in particular this places strong conditions on $L/k$ necessary in order for $c_{L/k}^*$
to be full and faithful.  

We collect several technical constructions and results in \S\ref{sec:constructions}.
In \S\ref{sub:models} and \S\ref{sub:stablemodels} we recall some definitions and facts about different model structures we use.
 With these preliminaries in order, the unstable and stable versions of $c^*_{L/k}$ are constructed in
\S\ref{sub:galois}.  In \S\ref{sub:betti}
we record the construction of and
some well-known results on the stable $C_2$-equivariant Betti
realization functor arising from a real embedding of fields.  In \S\ref{sub:change} we prove basic
compatibility results between $c_{L/k}^*$ and various change-of-group
and change-of-base functors.  Finally, in \S\ref{sub:betti2} we
study the effect of stable $C_2$-equivariant Betti realization on
motivic cohomology. In particular, we show that the Beilinson-Lichtenbaum conjectures can be rephrased for real closed subfields of $\R$
in terms of Bredon cohomology (\aref{thm:BL}) 
and we establish an equivariant version of a theorem of
Suslin-Voevodsky for torsion effective motives (\aref{thm:SV}).

\subsection{Relation to other work}
It is interesting to contrast the subject of this paper with Hu, Kriz,
and Ormsby's 
stable equivariant motivic homotopy theory \cite{HKO:equimot}. 
In that setup one studies smooth schemes equipped
with a $G$-action, $G$ a finite group. It should be emphasized that this group does not necessarily have any relationship
with the automorphisms of a field extension.  
In contrast, in the present work we
study the image of the stable $\Gal(L/k)$-equivariant homotopy category
inside the stable \emph{nonequivariant}
motivic homotopy category over $k$.  It would be interesting to  combine these  notions of
equivariance and geometry further by studying $(G,\Gal(L/k))$-homotopy inside of the
$G$-motivic homotopy category over $k$. 

\subsection{Notation and conventions} 
Throughout $k$ is a perfect field and $L/k$ is a finite Galois extension with Galois group $G$.  
For a finite group $G$ we write $\SH_{G}$ for the (genuine) stable equivariant homotopy category.
We write $\Sm/k$ for the category of smooth schemes of finite type
over a base field $k$ and we write $\SH_{k}$ for the stable motivic
homotopy category. We use the notation $[-,-]_G = \SH_G(-,-)$ and
$[-,-]_k = \SH_k(-,-)$ for morphism sets in respective stable homotopy
categories.  Our indexing convention for motivic spheres is that
$S^{a +b\alpha} := (S^{1})^{\wedge a}\wedge
(\A^{1}\setminus\{0\})^{\wedge b}$. When $G=C_{2}$ we write
$S^{\sigma}$ for the sign-representation sphere and set
$S^{a+b\sigma}:= (S^1)^{\wedge a}\wedge (S^{\sigma})^{\wedge b}$.  In
the special case $a=b=0$, we write $\SS_k$ and $\SS_G$ for the sphere
spectra in the motivic and equivariant categories, respectively.

For the sake of typographical simplicity, we do not use any special notational device for derived functors in \S\ref{sec:main} and \S\ref{sec:classical}, where we only work on the level of homotopy categories. In
\S\ref{sec:constructions} we work in both model categories and associated homotopy categories and in this section we use ``derived functor notation'' (i.e. $\L F$ and $\R F$ respectively for left and right derived functor of $F$).

\subsection*{Acknowledgements}
We are grateful to Marc Levine and Dan Dugger for spotting errors in previous drafts of this paper.
We thank Paul Arne {\O}stv{\ae}r, Kirsten Wickelgren, and the anonymous referee for helpful comments. We have also benefitted from the Algebraic Topology semester at MSRI in Spring 2014.  
The first author also thanks the MIT math department for generous hospitality during the preparation of this paper.
The second author gratefully acknowledges support from the NSF.

\section{Embedding theorem}\label{sec:main}
Let $L/k$ be a Galois extension of fields with Galois group $G$. As mentioned in the introduction, the functor $\Or_{G}\to \sm/k$ which  is defined on objects by $G/H\mapsto \spec(L^{H})$,  induces a functor 
$c^*_{L/k}:\SH_{G}\to \SH_{k}$ on stable homotopy categories. Details on this construction are given in \aref{sec:constructions}.

Our embedding result concerns certain completions of the functor $c^*_{L/k}$.
Recall that the $(p,\eta)$-completion $X^{\wedge}_{p,\eta}$ of a motivic spectrum is defined to be the Bousfield localization of $X$ at $\SS_{k}/(p,\eta):=\cofiber(S^{\alpha}\wedge \SS_k/p\to \SS_k/p)$. 
We have a motivic equivalence  
$(\SS_k)^\wedge_{p,\eta} \wkeq \holim \SS_{k}/(p^n,\eta^n)$.
Similarly, for a $C_2$-spectrum $Y$, define $Y^{\wedge}_{p,\eta}$ to be the Bousfield localization of $Y$ at the spectrum $\SS_{C_2}/(p,\eta)$.\footnote{The map $\eta:\SS^{\sigma}\to S^0$ in $\SH_{C_2}$ is the stable map induced by $\C^2-\{0\}\to\C P^2$.} We have an equivariant equivalence $(\SS_{C_2})^{\wedge}_{p,\eta} =\holim \SS_{C_2}/(p^n,\eta^n)$. 
Write $(\SH_{k})^{\wedge}_{p,\eta}\subseteq \SH_k$ and $(\SH_{C_{2}})^{\wedge}_{p,\eta}\subseteq \SH_{C_2}$ respectively for the full subcategories of $(p,\eta)$-complete objects. Note that these are triangulated subcategories. Write
$(c^*_{L/k})^{\wedge}_{p,\eta} := (-)^{\wedge}_{p,\eta}\circ c^*_{L/k}$. 
In this section we prove that if $k$ is a real closed field and $L=k[i]$ then 
$$
(c^*_{L/k})^{\wedge}_{p,\eta}:(\SH_{C_{2}})^{\wedge}_{p,\eta}\to (\SH_{k})^{\wedge}_{p,\eta}
$$  
is a full and faithful embedding for any prime $p$. 
Additionally if $\pi_{n}(\SS_k)_{\Q}=0$ for any $n>0$ (see \aref{Serrefiniteness}) then the functor 
$$
(c^*_{L/k})^{\wedge}_{\eta}:(\SH_{C_{2}})^{\wedge}_{\eta}
\to (\SH_{k})^{\wedge}_{\eta} 
$$ 
is full and faithful without $p$-completion.
This is proved in \aref{thm:textmain} and \aref{thm:textmain2}. 
The main step is to show that the $C_{2}$-equivariant Betti realization induces isomorphisms 
\begin{enumerate}
\item[(i)] $\RRe^{C_{2}}_{B,\phi}:  [S^n,(\SS_{k})^{\wedge}_{p,\eta}]_{k} \xrightarrow{\iso} 
[S^n,  (\SS_{C_{2}})^{\wedge}_{p,\eta}]_{C_{2}}$, and
\item[(ii)] $\RRe^{C_{2}}_{B,\phi}:  [\spec(L)_+\wedge S^n,(\SS_{k})_{p,\eta}^{\wedge}]_{k} \xrightarrow{\iso} 
[C_{2\,+}\wedge S^n,  (\SS_{C_{2}})^{\wedge}_{p,\eta}]_{C_{2}}$
\end{enumerate}
whenever there is a real embedding $\phi:k\hookrightarrow \R$.

\subsection{Completing at {$p$} and {$\eta$}} 
Let $p$ be a prime.
We analyze the image, under equivariant Betti realization, of the motivic Adams spectral sequence over a real closed subfield of $\R$. 
Let $\HH\ZZ/p$ denote the mod-$p$ motivic cohomology spectrum.  The
motivic Adams spectral sequence 
for $\SS_k$ arises as the totalization
spectral sequence of the semi-cosimplicial $\P^1$-spectrum with
$s$-th spectrum $(\HH\ZZ/p)^{\wedge s}$ and co-face maps induced by the
unit $\SS_k\to \HH\ZZ/p$.  We use the following specialization of a
theorem of P.~Hu, I.~Kriz, and the second author.

\begin{theorem}[{\cite[Theorem 1]{HKO}}]\label{thm:MASSconv}
Let $k$ be a real closed field, $L=k[i]$, and let $Y$ be either $\SS_k$ or $\spec(L)_{+}$. The motivic Adams spectral sequence
\[
  E_1^{s,t} = [S^{t}\wedge Y, (\HH\ZZ/p)^{\wedge s}]_{k}\implies 
  [S^{t-s}\wedge Y, (\SS_k)^\wedge_{p,\eta}]_{k}.
\]
is strongly convergent. If $p=2$ then 
$[S^{t-s}\wedge Y, (\SS_k)^\wedge_{2,\eta}]_{k} = [S^{t-s}\wedge Y, (\SS_k)^\wedge_2]_{k}$.
\end{theorem}
\begin{proof}
This is the weight zero portion of the $p$-primary motivic Adams spectral sequence
constructed in \cite{HKO} over $k$ (when $Y=\SS_k$) or over $L$ (when $Y=\spec(L)_{+}$).  The form of the $E_1$-page is immediate
from the totalization construction.  Convergence follows from
\cite[Theorem 1]{HKO}, which states that over a field $k$ of characteristic $0$, the Adams spectral sequence for a finite cell spectrum at $p$ converges to $(p,\eta)$-completions. Moreover, by loc.~cit.~ if $cd_2(k[i])<\infty$ then $(\SS_k)^{\wedge}_{2}\to (\SS_k)^\wedge_{2,\eta}$ induces an isomorphism on motivic homotopy groups.  Real closed fields satisfy $cd_2(k[i])<\infty$ and so we can indeed invoke \cite[Theorem 1]{HKO}.
\end{proof} 
 
We now turn to the $C_2$-equivariant Adams spectral sequence. This spectral sequence has
has been studied for $p=2$ by P. Hu and I. Kriz \cite{HK}, where it is shown that it converges to the $2$-completion. For odd $p$, the situation is a little different. The target of this spectral sequence is the $\HH\ul{\Z/p}$-nilpotent completion, which can be different than $p$-completion. We briefly recall its definition and construction and then show that it agrees with $(p,\eta)$-completion in general.

Bousfield's construction and discussion of the nilpotent completion and its relation to the Adams spectral sequence in \cite{Bousfield} applies as well to the equivariant setting. For a  concise recollection, 
see \cite[Section 6.7]{DI:MASS}  (the discussion of loc.~cit.~is tailored to the motivic setting but applies to the equivariant setting with evident modification). 
Let $E$ be a $C_2$-equivariant ring spectrum and define $\overline{E}$ to be the fiber of the unit map $\SS_{C_2}\to E$. For a spectrum $X$ we 
define 
$$
X_{s} = \overline{E}^{\wedge s}\wedge X\;\; \text{and}\;\; 
C_{s}=\cofiber(X_{s+1}\to X).
$$ 
There are maps $X_{s+1}\to X_{s}$, and hence maps $C_{s}\to C_{s-1}$ 
induced by $\overline{E}\to \SS_{C_2}$. 
The \textit{$E$-nilpotent completion} of $X$ is defined to be
$$
X^{\wedge}_{E}=\holim(C_{s}).
$$
Note that there are cofiber sequences 
$X_{\infty} \to X \to X^{\wedge}_{E}$,
where $X_{\infty} := \holim X_{s}$. 
The tower $\{C_{s}\}$ forms an $E$-nilpotent resolution of $X$, in the sense of \cite[Definition 5.6]{Bousfield}. The Tot-tower associated to the cosimplicial spectrum $E^{\wedge \bullet}\wedge X$ also forms an $E$-nilpotent resolution of $X$. 
The arguments of 
\cite[Proposition 5.8]{Bousfield} thus show 
that the homotopy limit of this Tot-tower is homotopic to $X^{\wedge}_{E}$.

Set $W_{s}:= E\wedge X_{s} = E\wedge \overline{E}^{\wedge s}\wedge X$, then $W_{s} = \cofiber(X_{s+1} \to X_{s})$. Note that we also have that
$\Sigma W_{s} =\cofiber(C_{s}\to C_{s-1})$.
By induction, each $C_s$ is $E$-local and hence so is  $X^{\wedge}_{E}$.  Therefore the map $\alpha:X\to X^{\wedge}_{E}$ factors through the Bousfield localization, $X\to L_{E}X \to X^{\wedge}_{E}$. 

\begin{lemma}[\cite{Bousfield}]
 The map $\beta:L_{E}X\to X^{\wedge}_{E}$ is an equivariant weak equivalence if and only if  $\alpha^{\wedge}_{E}:X^{\wedge}_{E}\to(X^{\wedge}_{E})^{\wedge}_{E}$ is an equivariant weak equivalence.
\end{lemma}
\begin{proof}
This is similar to  \cite[p. 273]{Bousfield}. 
  We have a retraction $E\wedge X\to E\wedge X^{\wedge}_{E}\to E\wedge X$ obtained from  $X^{\wedge}_{E}\to C_{0}=E\wedge X$ together with $E\wedge E\to E$. One finds that $X\to Y$ is an $E$-equivalence if and only if $X^{\wedge}_{E}\to Y^{\wedge}_{E}$ is an equivariant equivalence. 
  In particular, $X^{\wedge}_{E}\to (L_{E}X)^{\wedge}_{E}$
  is an equivariant equivalence. 
  The map $\beta$ is an equivalence if and only if it is an $E$-equivalence and so $\beta$ is an equivalence if and only if $(L_{E}X)^{\wedge}_{E}\wkeq X^{\wedge}_{E}$. This happens if and only if $\alpha^{\wedge}_{E}$ is an equivalence. 
\end{proof}

\begin{proposition}\label{prop:idem}
Let $R$ be a subring of $\mathbb{Q}$. Suppose that $E$ satisfies the condition that the geometric fixed points spectrum $\Phi^{K}(E)$, $K=\{e\}, C_2$ are $N_{K}$-connective for some $N_K$, and $H_{r}(\Phi^{K}(E))$ is a finitely generated $R$-module for all $r$. Let $X$ be an $C_2$-spectrum such that each $\Phi^{K}(X)$ is $M_K$-connective for some $M_K$.
Then $\alpha^{\wedge}_{E}:X^{\wedge}_{E}\xrightarrow{\wkeq} (X^{\wedge}_{E})^{\wedge}_{E}$ is an equivariant equivalence.
\end{proposition}
\begin{proof}
 There are functorial cofiber sequences $X_{\infty} \to X \to X^{\wedge}_{E}$. The result follows by showing  that  $(X_{\infty})_{\infty}\wkeq X_{\infty}$. 
 We claim that the map 
 $$
 (X_{\infty})_{s} = \overline{E}^{\wedge s}\wedge  \holim_{n}(\overline{E}^{\wedge n}\wedge X) \to \holim_{n}\overline{E}^{\wedge s+n}\wedge X 
 $$ 
 is an equivariant weak equivalence. There is a cofiber sequence $\holim_i Y_i\to \prod_iY_i\to \prod_iY_i$, and so it suffices to see that 
 $\overline{E}^{\wedge s}\wedge  \prod_n(\overline{E}^{\wedge n}\wedge X) \to \prod_{n}\overline{E}^{\wedge s+n}\wedge X$ is an equivariant weak equivalence.
 It follows from \cite[Thereom III.15.2]{Adams:blue} and \aref{lem:gfpprod} that this map is a weak equivalence on geometric fixed points as well as on the underlying spectrum and so it is an equivariant weak equivalence.
  The map $(X_{\infty})_{s}\to X_{\infty}$ is thus an equivariant weak equivalence and so taking homotopy limits we have that  $(X_{\infty})_{\infty}\wkeq X_{\infty}$ as desired. 
\end{proof}

For a $C_2$-spectrum $E$, we write $\pi_n^{C_2}(E) = [S^n, E]_{C_2}$ for the $n$th stable equivariant homotopy group.
\begin{lemma}\label{lem:gfpprod}
 Let $Y_i$, $i\in \mathbb{N}$, be $C_2$-spectra. Suppose that there is an integer $N$ so that the underlying spectrum of $Y_i$ is $N$-connective. Then
 $\Phi^{C_2}(\prod_i Y_i) \wkeq \prod_i \Phi^{C_2}(Y_i)$. 
\end{lemma}
\begin{proof}

The geometric fixed points of $X$ are equal to $(\widetilde{E}C_2\wedge X)^{C_2}$. We have an equivariant equivalence $\widetilde{E}C_2 \wkeq \colim_{k}S^{k\sigma}$.
Note that
$S^{k\sigma}\wedge (\prod Y_i) \wkeq \prod (S^{k\sigma} \wedge Y_i)$, since $S^{k\sigma}$ is dualizable. We thus need to see that the map
\begin{equation}\label{eqn:gfpcomm}
 \colim_{k} \prod (S^{k\sigma} \wedge Y_i) \to\prod \colim_k (S^{k\sigma} \wedge Y_i) 
 \end{equation}
 induces an isomorphism on  $\pi^{C_2}_n$ for all $n$.
Consider the cofiber sequence 
$$
C_{2+}\wedge S^{k\sigma} \wedge Y_i \to S^{k\sigma}\wedge Y_i \to S^{(k+1)\sigma} \wedge Y_{i}.
$$ 
Since each $Y_i$ is $N$-connective, for a fixed $n$, there is an integer $s$ such that 
$\pi_{n}^{C_2}(C_{2+}\wedge S^{k\sigma} \wedge Y_i)$  vanishes for all $k>s$. This implies that 
$\pi_{n}^{C_2}(S^{k\sigma}\wedge Y_i ) = \pi_{n}^{C_2}( S^{(k+1)\sigma} \wedge Y_{i})$. 
We thus have that  the (\ref{eqn:gfpcomm}) induces an isomorphism on $\pi_n^{C_2}$, as desired.
\end{proof}

Recall that $\ul{\pi}_n(X)$ denotes the Mackey functor homtopy groups. Say that a $C_2$-spectrum $X$ is \emph{$n$-connective} $\ul{\pi}_k(X) = 0$ for $k<n$. Say that a map $f:X\to Y$ is an $n$-equivalence if its cofiber is $n+1$-connective. For the following lemma, note that since $\eta$ is zero on $\HH\ul{\Z}$, the map $\HH\ul{\Z}\to \HH\ul{\Z}/\eta$ is split.

\begin{lemma}\label{lem:oneeq}
 Let $X$ be a connective $C_2$-spectrum and $p$ an odd prime. Then the unit map
 $X/(p^{s},\eta^{t})\to \HH\ul{\Z}\wedge X/(p^{s},\eta^{t})$
 and the map $X/(2^s,\eta^t)\to \HH\ul{\Z}\wedge X/2^s$ are $1$-equivalences for any integer $s,t\geq 1$. 
 \end{lemma}
\begin{proof}
There are cofiber sequences $X/(p,\eta^t)\to X/(p^s,\eta^t) \to X/(p^{s-1},\eta^{t})$ and 
a similar one for quotients by powers of $\eta$. An inductive argument shows that it suffices to consider the case $s=1$, $t=1$.
It suffices to consider $X=\SS_{\C_2}$, in which case a straightforward calculation shows that 
$\SS_{\C_2}/(p,\eta)\to \HH{\ul{\Z}}\wedge \SS_{\C_2}/(p,\eta)$ and $\SS_{\C_2}/(2,\eta)\to \HH{\ul{\Z}}\wedge \SS_{\C_2}/2$ are $1$-equivalences. 
\end{proof}

\begin{proposition}\label{lem:otoh}
 Let $X$ be a connective, $C_2$-spectrum and $p$ a prime. Suppose that $X$ satisfies the condition  that both multiplication by $p^{s}$ and by $\eta^{t}$ are equal to zero on $X$, for some integers $s,t\geq 1$. 
 Then $X\to X^{\wedge}_{\HH{\ul{\Z/p}}}$ is an equivariant equivalence.
\end{proposition}
\begin{proof}
We treat the case of an odd prime explicitly, $p=2$ is similar.
First note that if $p^{s}$ and $\eta^{t}$ both act by zero on a spectrum $Z$, then  $Z$ is a summand of $Z/(p^s,\eta^t)$. 
Note as well that the previous lemma implies that if $Y$ is $n$-connective, for $n\geq 0$, then $Y\to \HH\ul{\Z}\wedge Y$ is an $n+1$-equivalence.
We inductively define cofiber sequences $X_{i+1}\to X_{i}\to K_{i}$ 
by letting $X_{0}:=X$ and  $K_{i}:=\HH\ul{\Z}\wedge X_i$. We claim that $X_i$ is $i$-connective for all $i$. Indeed if $X_i$ is $i$-connective then $X_i/(p^s,\eta^t)\to K_i/(p^s,\eta^t)$ is an $i+1$-equivalence. But this map contains $X_i\to K_i$ as a summand and so it is an $i+1$-equivalence as well which implies that $X_{i+1}$ is $i+1$-connective.

Write $C_i=\cofiber(X_{i+1}\to X)$.
The tower $\{C_i\}$ is an $\HH\ul{\Z}$-nilpotent resolution of $X$ and we claim that it is in fact an $\HH\ul{\Z/p}$-nilpotent resolution.  
If $N$ is $\HH\ul{\Z/p}$-nilpotent, 
then $\colim_i[Y\wedge X_i, N]_{C_2}=0$, where $Y=S^{0}$ or $C_2$. 
It remains to see that 
the $C_i$ are $\HH{\Z/p}$-nilpotent.  
This is seen by induction by noting the $K_i$ are $\HH\ul{\Z/p}$-nilpotent since there is a splitting of $K_i\to \HH\ul{\Z/p}\wedge K_i$  as follows.
We have 
$\HH\ul{\Z/p^N}\wedge K_i = (\HH\ul{\Z}\wedge X_i)\vee (\Sigma\HH\ul{\Z}\wedge X_i)$ and so a splitting of $K_i\to \HH\ul{\Z/p}\wedge K_i$ is obtained via the composition 
$$
\HH\ul{\Z/p}\wedge K_i\to \HH\ul{\Z/p^N}\wedge K_i\to \HH\ul{\Z}\wedge  X_i=K_i.
$$
Since $X_i$ is $i$-connective, we have that $\holim_{i}X_{i}\wkeq \ast$ and therefore $X = X^{\wedge}_{\HH\ul{\Z/p}}$ as desired.
\end{proof}

\begin{proposition}\label{prop:nilnil}
 Let $X$ be a connective $C_2$-spectrum and $p$ a prime. Then there is a natural equivariant equivalence 
 $X^{\wedge}_{\HH{\ul{\Z/p}}} \wkeq X^{\wedge}_{p,\eta}$. 
\end{proposition}
\begin{proof}
 The map $X\to X^{\wedge}_{p,\eta}$ is an $\SS_{C_2}/(p,\eta)$-equivalence, and therefore an $\HH\ul{\Z/p}$-equivalence. It follows that $X^{\wedge}_{\HH\ul{\Z/p}}\to (X^{\wedge}_{p,\eta})^{\wedge}_{\HH\ul{\Z/p}}$ is 
 an equivariant weak equivalence. 
 On the other hand, \aref{lem:otoh} implies that $X/(p^{n},\eta^{n})\to (X/(p^n,\eta^n))^{\wedge}_{\HH\ul{\Z/p}}$ is an equivariant equivalence for all $n$. Therefore we have that $X^{\wedge}_{p,\eta}\to \holim_{n}(X/(p^n,\eta^n))^{\wedge}_{\HH\ul{\Z/p}} \wkeq (X^{\wedge}_{p,\eta})^{\wedge}_{\HH\ul{\Z/p}}$. 
  \end{proof}

 \begin{lemma}[{\cite[Corollary 6.47]{HK}}]\label{lem:nil2nil}
  Let $X$ be a connective $C_2$-spectrum. There is a natural equivariant equivalence $X^{\wedge}_{\HH{\ul{\Z/2}}} \wkeq X^{\wedge}_{2}$.
 \end{lemma}
\begin{proof}
 By the previous proposition it suffices to show that $i:X^{\wedge}_{2}\to X^{\wedge}_{2,\eta}$
is an equivariant equivalence. The map $i$ is an equivalence  after forgetting the action, so it suffices to show that it induces an isomorphism on $\pi^{C_2}_n$. Write $F$ for the homotopy fiber of $i$. Note that $\eta:S^{\sigma}\wedge F\to F$ is an equivariant equivalence. Note as well that $\rho:F\to S^{\sigma}\wedge F$ is a weak equivalence, since $F$ is nonequivariantly contractible. 
We have the relation $\eta^{2}\rho= -2\eta $. In particular, we find that $2$ is an equivalence on $F$ and so $F/2^s\wkeq \ast$ for all $s$. Since $F$ is $2$-complete, we have $F\wkeq \ast$.
 \end{proof}

Fix an embedding $\phi:k\hookrightarrow \R$
and consider the resulting $C_2$-equivariant Betti realization 
$\RRe_{B,\phi}^{C_2}:\SH_k\to \SH_{C_2}$ (see
\aref{sub:betti} for details).  By \aref{thm:bettibredon},
the equivariant Betti realization takes the motivic cohomology spectrum $\HH\ZZ/p$ to the Bredon cohomology spectrum
$\HH\ul{\ZZ/p}$ associated to the
constant Mackey functor $\ul{\ZZ/p}$.  Since $\RRe_{B,\phi}^{C_2}$ is
symmetric monoidal and takes the unit for $\HH\ZZ/p$ to the unit for
$\HH\ul{\ZZ/p}$, we see that $\RRe_{B,\phi}^{C_2}$ takes the
semi-cosimplicial $\P^1$-spectrum $(\HH\ZZ/p)^{\wedge \bullet}$ to the
semi-cosimplicial $C_2$-spectrum 
$(\HH\ul{\ZZ/p})^{\wedge \bullet}$.  The totalization spectral sequence for this latter
object is the \emph{$C_2$-equivariant Adams spectral sequence}, which
has been studied by P. Hu and I. Kriz \cite{HK} and the case $p=2$ of the following theorem is  \cite[Corollary 6.47]{HK}.

\begin{theorem}\label{thm:C2ASSconv}
Let $Y$ be either $\SS_{C_2}$ or $C_{2\,+}$. The $C_2$-equivariant Adams spectral sequence  
\[
  E_1^{s,t} = [S^t\wedge Y,  (\HH\ul{\ZZ/p})^{\wedge s}]_{C_{2}}
  \implies 
  [S^{t-s}\wedge Y, (\SS_{C_2})^\wedge_{p,\eta}]_{C_{2}}
\]
is strongly convergent. If $p=2$ then $(\SS_{C_2})^{\wedge}_{2,\eta} = (\SS_{C_2})^{\wedge}_2$.
\end{theorem}
\begin{proof}
The spectral sequence associated to the Tot-tower of the semi-cosimplicial object 
$Y\wedge(\HH\ul{\ZZ/p})^{\wedge s}$ agrees with the spectral sequence associated to the 
tower  $\{Y\wedge \overline{\HH{\ul{\Z/p}}}^{\wedge s}\}$. This in turn agrees with the spectral sequence with $Y$ replaced by $Y^{\wedge}_{\HH\ul{\Z/p}}$. This spectral 
sequence converges to $Y^{\wedge}_{\HH\ul{\Z/p}}$ since we have that
$\holim_s Y^{\wedge}_{\HH{\ul{\Z/p}}}\wedge \overline{\HH{\ul{\Z/p}}}^{\wedge s} \wkeq \ast$ (as $Y^{\wedge}_{\HH\ul{\Z/p}}\wkeq (Y^{\wedge}_{\HH\ul{\Z/p}})^{\wedge}_{\HH\ul{\Z/p}}$).
Together with the identifications of \aref{prop:nilnil} and \aref{lem:nil2nil}, this establishes the result.
  \end{proof}

By comparing these two Adams spectral sequences, we obtain the following
result.

\begin{proposition}\label{prop:BettiSubR}
Let $k$ be real closed, set $L=k[i]$, and let
$\phi:k\hookrightarrow\RR$ be an embedding of fields. Then the induced maps
\begin{enumerate}
\item[(i)] $\RRe_{B,\phi}^{C_2}:[S^n,(\SS_k)^\wedge_{p,\eta}]_k\xrightarrow{\iso} [S^n,(\SS_{C_2})^\wedge_{p,\eta}]_{C_2}$, and
\item[(ii)] $\RRe_{B,\phi}^{C_2}:[\spec(L)_{+}\wedge S^n,(\SS_k)^\wedge_{p,\eta}]_k\xrightarrow{\iso} [C_{2\,+}\wedge S^n,(\SS_{C_2})^\wedge_{p,\eta}]_{C_2}$
\end{enumerate}
are isomorphisms for any $n\in \Z$. For $p=2$, the induced maps $[S^n,(\SS_k)^\wedge_{2}]_k\xrightarrow{\iso} [S^n,(\SS_{C_2})^\wedge_{2}]_{C_2}$ and 
$[\spec(L)_{+}\wedge S^n,(\SS_k)^\wedge_{2}]_k\xrightarrow{\iso} [C_{2\,+}\wedge S^n,(\SS_{C_2})^\wedge_{2}]_{C_2}$ are isomorphisms.
\end{proposition}
\begin{proof}
We have already noted that 
$\RRe_{B,\phi}^{C_2}(\HH\ZZ/p)^{\wedge s}\simeq
\HH\ul{\ZZ/p}^{\wedge s}$, and that we have a map of Adams spectral
sequences.  
The computation of the motivic Steenrod algebra
\cite{Voevodsky:reducedpower, Voevodsky:Zl} shows that we have a decomposition $\HH\Z/p\wedge \HH\Z/p \wkeq \vee \Sigma^{p_{i}+q_{i}\alpha}\HH\Z/p$ for appropriate $(p_{i},q_{i})$ which in particular satisfy $q_{i}\geq 0$.
It follows from \aref{thm:SV} that the equivariant Betti realization induces an isomorphism on the weight zero $E_1$-page of the Adams spectral sequences.
By \aref{thm:MASSconv} and
\aref{thm:C2ASSconv}, the proposition follows.
\end{proof}

\subsection{Rational homotopy groups}\label{sub:rat}
For a (motivic or equivariant) spectrum $X$ we write $X_{\Q}$ for the Bousfield localization at $\MM \Q$, the rational Moore spectrum.
If $Y$ is a compact spectrum, then  $[Y,X_{\Q}] = [Y,X]\otimes \Q$.

The homotopy groups of the equivariant rational sphere spectrum are rather simple.
\begin{proposition}\label{prop:equivariantQ}
 The homotopy groups of the rational $C_2$-sphere are
 $\pi_{0}(\SS_{C_2})_{\Q} = \Q\oplus \Q$ and 
 $\pi_{n}(\SS_{C_2})_{\Q} = 0$ for any integer $n\neq 0$.
\end{proposition}
\begin{proof}
This follows immediately from the well known fact (see \emph{e.g.}~\cite[Corollary A.6]{GM:Tate}) that for any finite group, $(\SS_{G})_{\Q}$ is weakly equivalent to $\HH\ul{A_{\Q}}$, the Eilenberg-MacLane spectrum associated to the rational Burnside Mackey functor.
\end{proof}

Conjecturally the higher homotopy groups of the sphere also vanish.

\begin{conjecture}[Motivic Serre finiteness]\label{Serrefiniteness}
Let $k$ be a field. Then $\pi_{n}(\SS_{k})_{\Q} = 0$ for $n>0$.
\end{conjecture}

\begin{definition}
We say that a field $k$ \emph{has motivic Serre finiteness} if \autoref*{Serrefiniteness} holds over $k$.
\end{definition}

Rationally (in fact already when $2$ is inverted) there are orthogonal idempotents
$\epsilon_{+} = (\epsilon - 1)/2$ and $\epsilon_{-}= (\epsilon + 1)/2$ acting on $(\SS_{k})_{\Q}$, obtained from $\epsilon\in\pi_{0} \SS_{k}$.\footnote{Recall that $\epsilon$ is the stable map induced by the permutation $\A^1\setminus\{0\}\wedge \A^1\setminus\{0\} \iso \A^1\setminus\{0\}\wedge \A^1\setminus\{0\}$.} 
We thus obtain a rational decomposition of the sphere spectrum
$(\SS_{k})_{\Q} = (\SS_{k})_{\Q}^{+}\vee (\SS_{k})^{-}_{\Q}$ in which the factors correspond respectively to  inverting $\epsilon_+$ and $\epsilon_{-}$. It follows from
Morel's description \cite{Morel:stable} of $(\SS_{k})_{\Q}^{+}$ as the rational motivic cohomology spectrum $\HH\Q$ 
(see \cite[Theorem 16.2.13]{CD:motives}) 
that $k$ has motivic Serre finiteness whenever $-1$ is sum of squares in $k$ (in which case
$(\SS_{k})^{-}_{\Q}$ vanishes). Morel \cite{Morel:stable} also conjectures a description of $(\SS_{k})_{\Q}^{-}$ which would imply that motivic Serre finiteness holds in general.

\begin{proposition}\label{prop:rational}
 Let $k$ be a real closed field, set $L=k[i]$, and let
 $\phi:k\hookrightarrow \R$ be an embedding. Assume that $k$ has motivic Serre finiteness.
 Then the maps
 \begin{enumerate}
 \item[(i)] $\RRe^{C_{2}}_{B,\phi}:  [S^n, (\SS_k)_{\Q}]_{k} \xrightarrow{\iso} 
[S^n,  (\SS_{C_2})_{\Q}]_{C_{2}}$, and
\item[(ii)] $\RRe^{C_{2}}_{B,\phi}:  [\spec(L)_{+}\wedge S^n, (\SS_k)_{\Q}]_{k} \xrightarrow{\iso} 
[C_{2\,+}\wedge S^n,  (\SS_{C_2})_{\Q}]_{C_{2}}$
\end{enumerate}
are isomorphisms for any $n\in \Z$. 
\end{proposition}
\begin{proof}
Since $\RRe^{C_{2}}_{B,\phi}\circ c^{*}_{L/k} = \id$, we know that the
map of the proposition is surjective. Since $GW(k)= \Z\oplus\Z$ and
$GW(L)= \Z$ for any real closed field $k$, it follows that the first map is an isomorphism in degree zero. By the previous propositions, these groups are zero in all other degrees.
\end{proof}

Write $\epsilon\in\pi_{0}^{C_2}(S^0)$ for the stable map induced by the permutation $S^{\sigma}\wedge S^{\sigma}\to S^{\sigma}\wedge S^{\sigma}$. As in the motivic setting, once $2$ is invertible there are idempotents $\epsilon_+=(\epsilon-1)/2$ and $\epsilon_-=(\epsilon +1)/2$ that induce a splitting $\SS_{C_2}[1/2]=\SS_{C_2}[1/2]^{+}\vee \SS_{C_2}[1/2]^{-}$.

\begin{lemma}\label{lem:eta12}
 Let $k$ be a field and $X$ and object of $\SH_{k}$. Then 
 $(X[1/2])^{\wedge}_{\eta} = (X[1/2])^{+}$. Similarly if $W$ is a $C_2$-spectrum, then $(W[1/2])^{\wedge}_{\eta}= W[1/2]^{+}$.
\end{lemma}
\begin{proof}
We have that $(X[1/2])^{\wedge}_{\eta} = (X[1/2]^+)^{\wedge}_{\eta}\vee (X[1/2]^-)^{\wedge}_{\eta}$.
From the relation $\epsilon\eta = \eta$, we find that $\eta:S^{\alpha}\wedge X[1/2]^+\to X[1/2]^+$ is zero and hence $X[1/2]^+$ is $\eta$-complete. On the other hand $\eta:S^{\alpha}\wedge X[1/2]^-\to X[1/2]^-$ is an equivalence and so $(X[1/2]^-)^{\wedge}_{\eta}\wkeq \ast$. It follows that $(X[1/2])^{\wedge}_{\eta} = (X[1/2])^{+}$ as desired. A similar analysis applies in the equivariant setting.
\end{proof}

\begin{corollary}\label{lem:etaplus}
Let $k$ be a real closed field, set $L=k[i]$, and let
 $\phi:k\hookrightarrow \R$ be an embedding. If $X$ is in $\SH_k$ and satisfies the condition that $\RRe^{C_2}_{B,\phi}:\pi_n(X_{\Q})\to \pi_n(\RRe^{C_2}_{B,\phi}(X_{\Q}))$ is an isomorphism, then $\RRe^{C_2}_{B,\phi}:\pi_n((X_{\Q})^{\wedge}_{\eta})\to \pi_n(\RRe^{C_2}_{B,\phi}(X_{\Q})^{\wedge}_{\eta})$ is also an isomorphism.
\end{corollary}
\begin{proof} 
The map  
$\RRe^{C_2}_{B,\phi}:\pi_n(X_{\Q})\to \pi_n(\RRe^{C_2}_{B,\phi}(X_{\Q}))$ is a direct sum of maps  
$$
(\RRe^{C_2}_{B,\phi})^{+}\oplus (\RRe^{C_2}_{B,\phi})^- :\pi_n(X_{\Q}^+)\oplus \pi_n(X_{\Q}^-)\to \pi_n(\RRe^{C_2}_{B,\phi}(X_{\Q})^+)\oplus \pi_n(\RRe^{C_2}_{B,\phi}(X_{\Q})^-).
$$
The result will thus follows from the previous lemma.
\end{proof}

\subsection{Full and faithful embedding }
We now assemble the previous computations to deduce our main theorem. 

\begin{proposition}\label{prop:Bettiintegral}
 Let $k$ be a real closed field, set $L=k[i]$, and let
 $\phi:k\hookrightarrow \R$ be an embedding. Assume that $k$ has motivic Serre finiteness.
  Then
 \begin{enumerate}
\item[(i)] $\RRe^{C_{2}}_{B,\phi}:  [S^n,(\SS_{k})^{\wedge}_{\eta}]_{k} \xrightarrow{\iso} 
[S^n,  (\SS_{C_{2}})^{\wedge}_{\eta}]_{C_{2}}$, and
\item[(ii)] $\RRe^{C_{2}}_{B,\phi}:  [\spec(L)_+\wedge S^n,(\SS_{k})^{\wedge}_{\eta}]_{k} \xrightarrow{\iso} 
[C_{2\,+}\wedge S^n,  (\SS_{C_{2}})^{\wedge}_{\eta}]_{C_{2}}$
\end{enumerate}
are isomorphisms for all $n\in\Z$. 
\end{proposition}
\begin{proof}
By  \cite[Appendix A]{OrmsbyOstvaer:pi1} there is a homotopy cartesian square in $\SH_{k}$
$$
\xymatrix{
\SS_{k} \ar[r]\ar[d] & \prod_{p}(\SS_{k})^{\wedge}_{p} \ar[d] \\
(\SS_{k})_{\Q} \ar[r] & \prod_{p}((\SS_{k})^{\wedge}_{p})_{\Q}
}
$$  
where the products are over prime integers $p$. There is a similar equivariant arithmetic fracture square in $\SH_{C_2}$\footnote{The authors do not know a handy reference for this equivariant arithmetic fracture square, but standard techniques adapt to produce it.  For instance, the proof giving the motivic arithmetic fracture square in \cite[Appendix A]{OrmsbyOstvaer:pi1} works almost verbatim in the equivariant setting.}. 
Taking the $\eta$-completion of this square yields the homotopy cartesian square
$$
\xymatrix{
(\SS_{k})^{\wedge}_{\eta} \ar[r]\ar[d] & \prod_{p}(\SS_{k})^{\wedge}_{p,\eta} \ar[d] \\
[(\SS_{k})_{\Q}]^{\wedge}_{\eta} \ar[r] & [\prod_{p}((\SS_{k})^{\wedge}_{p})_{\Q}]^{\wedge}_{\eta}
}
$$ 
and similarly in $\SH_{C_2}$.
 Since $X_\Q\to (X^\wedge_{\eta})_\Q$ is a filtered colimit of $\SS_k/\eta$-equivalences, it is itself an $\SS_k/\eta$-equivalence. It follows that  
$[(\SS_{k})^{\wedge}_{p})_{\Q}]^{\wedge}_{\eta} \wkeq [(\SS_{k})^{\wedge}_{p,\eta})_{\Q}]^{\wedge}_{\eta}$ and similarly for the $C_2$-equivariant case.
The square obtained by applying $\RRe^{C_{2}}_{B,\phi}$ to the above square maps to the equivariant arithmetic fracture square.
We thus obtain a comparison diagram of associated  long exact sequences.
The proposition thus follows from \aref{prop:BettiSubR}, \aref{prop:rational}, \aref{lem:etaplus} and the five lemma.
\end{proof}

We now turn our attention to $(c_{L/k}^*)^{\wedge}_{\eta}$.  Write $\eta_L$ for the map $c^*_{L/k}(\eta):S^L\to S^0$.  
\begin{lemma}\label{lem:commeta}
Let $k$ be a  real closed field and let $L = k[i]$.  Then the canonical map
$c_{L/k}^*(\SS_{C_2})^{\wedge}_{\eta}\to (\SS_{k})^{\wedge}_{\eta_L}$ 
is also an equivalence. The canonical map $(\SS_{k})^{\wedge}_{\eta}\to 
(\SS_{k})^{\wedge}_{\eta,\eta_{L}}$ is an equivalence. In particular 
$(c_{L/k}^*)^{\wedge}_{\eta}((\SS_{C_2})^{\wedge}_{\eta}) \wkeq (\SS_{k})^{\wedge}_{\eta}$.
\end{lemma}
\begin{proof}
 We show that the first equivalence holds for the $2$-complete sphere and for spectra on which $2$ is invertible. A comparison of fracture squares then implies the result. 
 First note that $c_{L/k}^*((\SS_{C_2})^{\wedge}_{2,\eta}) = (\SS_k)^{\wedge}_{2}$ by \aref{lem:nil2nil} and \aref{prop:nilnil}. The map $\eta_L$ induces the $\HH\Z/2$-module map $\eta_L:S^{L}\wedge \HH\Z/2 \to \HH\Z/2$. The group of $\HH\Z/2$-module maps from $S^{L}\wedge \HH\Z/2$ to $\HH\Z/2$ is identified with the group $[S^L, \HH\Z/2]_k=0$. Thus $\eta_L$ acts by zero on any $\HH\Z/2$-module and so any $\HH\Z/2$-module is $\eta_L$-complete. It follows that $(\SS_k)^{\wedge}_{\HH\Z/2}$ is $\eta_L$ complete which by \aref{thm:C2ASSconv} implies that 
 $(\SS_k)^{\wedge}_{2}$ is $\eta_L$-complete. 
  Now if $2$ is invertible on $X$ then we have 
  $X^{\wedge}_{\eta} = X[\epsilon_+^{-1}]$ and $c^*_{L/k}(X)^{\wedge}_{\eta_L} = c_{L/k}^{*}(X)[(\epsilon_L)_+^{-1}]$ by \aref{lem:eta12}. Since $c_{L/k}^{*}(X[\epsilon_+^{-1}]) = c_{L/k}^{*}(X)[(\epsilon_L)_+^{-1}]$, we have established the first equivalence.

  For the second equivalence, we compare the applications of $(-)^{\wedge}_{\eta}$ and $(-)^{\wedge}_{\eta,\eta_{L}}$ to the arithmetic fracture square. 
  Since $(\SS_{k})^{\wedge}_{p,\eta}\wkeq (\SS_{k})^{\wedge}_{\HH\Z/p}$ and $[S^L, \HH\Z/p]_k=0$ we find that  $(\SS_{k})^{\wedge}_{p,\eta}$ is $\eta_L$-complete. Let $X$ be an object of $\SH_{k}$. By \aref{lem:eta12} we have $(X_{\Q})^{\wedge}_{\eta} \wkeq X_{\Q}^{+}$ and  
    by \cite[Theorem 16.2.13]{CD:motives} we have $X_{\Q}^{+}\wkeq X\wedge \HH\Q$. Since $[S^L, \HH\Q]_k=0$, we find that $(X_{\Q})^{\wedge}_{\eta}$ is $\eta_L$-complete. This implies the second equivalence.
\end{proof}
  
We now convert our analysis of $\RRe_{B,\phi}^{C_2}$ to $c^*_{L/k}$ using a limit argument which is  a modification of the one used in
\cite[Lemma 6.6]{Levine:comparison} to the case of real closed fields. 
  
\begin{proposition}\label{prop:csphere}
Let $k$ be a real closed field and set $L=k[i]$. Assume that $k$ has motivic Serre finiteness. Then for any $n\in \Z$, the maps
\begin{enumerate}
\item[(i)]  $(c^*_{L/k})^{\wedge}_{\eta}:[S^n,(\SS_{C_2})^{\wedge}_{\eta}]_{C_2}
\xrightarrow{\iso} [S^n, (c^*_{L/k})^{\wedge}_{\eta}((\SS_{C_2})^{\wedge}_{\eta})]_k$, and
\item[(ii)] $(c^*_{L/k})^{\wedge}_{\eta}:
[C_{2\,+}\wedge S^n, (\SS_{C_2})^{\wedge}_{\eta}]_{C_2}\xrightarrow{\iso} 
[\spec(L)_+\wedge S^n, (c^*_{L/k})^{\wedge}_{\eta}((\SS_{C_2})^{\wedge}_{\eta})]_k$
\end{enumerate}
are isomorphisms.
For any prime $p$, the maps 
$[S^n,(\SS_{C_2})^{\wedge}_{p,\eta}]_{C_2}\xrightarrow{\iso} [S^n, (c^*_{L/k})^{\wedge}_{p,\eta}((\SS_{C_2})^{\wedge}_{p,\eta})]_k$, and
 $[C_{2\,+}\wedge S^n, (\SS_{C_2})^{\wedge}_{p,\eta}]_{C_2}\xrightarrow{\iso} 
[\spec(L)_+\wedge S^n, (c^*_{L/k})^{\wedge}_{p,\eta}((\SS_{C_2})^{\wedge}_{p,\eta})]_k$
for $(p,\eta)$-completed spheres are always isomorphisms. For $p=2$, the maps 
$[S^n,(\SS_{C_2})^{\wedge}_{2}]_{C_2}\xrightarrow{\iso} [S^n, (\SS_{k})^{\wedge}_{2}]_k$, and
 $[C_{2\,+}\wedge S^n, (\SS_{C_2})^{\wedge}_{2}]_{C_2}\xrightarrow{\iso} 
[\spec(L)_+\wedge S^n, (\SS_{k})^{\wedge}_{2})]_k$ are isomorphisms.
\end{proposition}
\begin{proof}
If there is an embedding $\phi:k\subseteq \R$, then this is a direct consequence of \aref{prop:BettiSubR}, \aref{prop:Bettiintegral}, and \aref{lem:commeta} and the relation 
$\RRe_{B,\phi}^{C_2}\circ
c_{L/k}^*\iso \id$. We treat the case of the $\eta$-completed spheres below, the case of $(p,\eta)$-completion holds verbatim.

As $k$ is real closed, $L$ is algebraically closed.  We may express
$L$ as the union $\bigcup_{\alpha\in A} L_\alpha$ of algebraically closed
subfields $L_\alpha\subset L$ of finite transcendence degree over
$\QQ$ indexed by a well-ordered set $A$.  Consider the fields
$k_\alpha = L_\alpha\cap k$.  We claim that the $k_\alpha$ are
isomorphic to real closed subfields of $\RR$.  If this is the case,
then
\[
  \colim_\alpha [S^n,(c^*_{L_\alpha/k_\alpha})^{\wedge}_{\eta}((\SS_{C_2})^{\wedge}_{\eta})]_{k_\alpha} \,\,\,\text{and}\,\,\, \colim_\alpha [\spec(L_{\alpha})_+\wedge S^n,(c^*_{L_{\alpha}/k_{\alpha}})^{\wedge}_{\eta}((\SS_{C_2})^{\wedge}_{\eta})]_{k_\alpha}
\]
are colimits of abelian groups with constant values
$[S^n, (\SS_{C_2})^{\wedge}_{\eta}]_{C_2}$ and $[C_{2\,+}\wedge S^n, (\SS_{C_2})^{\wedge}_{\eta}]_{C_2}$ respectively, by the observation in the first paragraph.  Since it is
obvious that $k = \bigcup_\alpha k_\alpha$, 
using essentially smooth base change 
\cite[Lemma A.7]{Hoyois:cob2mot} we conclude  
that these colimits are respectively isomorphic to
$[S^n,(c^*_{L/k})^{\wedge}_{\eta}((\SS_{C_2})^{\wedge}_{\eta})]_k$ and $[\spec(L)_+\wedge S^n,(c^*_{L/k})^{\wedge}_{\eta}((\SS_{C_2})^{\wedge}_{\eta})]_k$.  Thus we may now conclude that the maps
$[S^n, (\SS_{C_2})^{\wedge}_{\eta}]_{C_2}\to [S^n, (c^*_{L/k})^{\wedge}_{\eta}((\SS_{C_2})^{\wedge}_{\eta})]_k$  and $[C_{2\,+}\wedge S^n,(\SS_{C_2})^{\wedge}_{\eta}]_{C_2}\to [\spec(L)_+\wedge S^n, (c^*_{L/k})^{\wedge}_{\eta}((\SS_{C_2})^{\wedge}_{\eta})]_k$ are isomorphisms for
all real closed fields.

It remains to verify the claim that each $k_\alpha$ is isomorphic to a
real closed subfield of $\RR$.  Since $L_\alpha$ is algebraically
closed and $[L_\alpha:k_\alpha]=2$, the Artin-Schreier theorem implies
that $k_\alpha$ is real closed.  Fix $k_\alpha$ and choose a 
transcendence basis $x_1,\ldots,x_n$ of $k_\alpha$ over $\QQ$ in which
each $x_i$ is positive in $k_\alpha$.  By
sending each $x_i$ to a positive transcendental real number, we
produce an order embedding of $\QQ(x_1,\ldots,x_n)$ into $\RR$.  Since
$k_\alpha/\QQ(x_1,\ldots,x_n)$ is a union of finite extensions of ordered
fields, \cite[Proposition VIII.2.16]{Lam:intro} implies that there is
an embedding $k_\alpha\hookrightarrow \RR$, as desired.
\end{proof}

We are now ready to prove our main theorem.  Recall that a \textit{localizing subcategory} $\mcal{E}$ of a triangulated category $\mcal{T}$ is a full triangulated subcategory, containing all direct summands of its objects and closed under arbitrary coproducts.  
 
\begin{theorem}\label{thm:textmain}
Let $k$ be a real closed field and let $L=k[i]$ be its algebraic closure. Assume that $k$ has motivic Serre finiteness.
Then
$$
(c_{L/k}^*)^{\wedge}_{\eta}:(\SH_{C_2})^{\wedge}_{\eta}\to (\SH_k)^{\wedge}_{\eta}
$$
is a full and faithful embedding. 
\end{theorem}
\begin{proof}
 Consider the subcategory $\mcal{C}\subseteq (\SH_{C_2})^{\wedge}_{\eta}$ whose objects are $\eta$-complete $C_2$-equivariant spectra
$X$ such that $(c^*_{L/k})^{\wedge}_{\eta}:[S^{n}, X]_{C_2}\to [S^{n}, c^*_{L/k}(X)^{\wedge}_{\eta}]_{k}$ 
and $(c^*_{L/k})^{\wedge}_{\eta}:[C_{2\,+}\wedge S^{n}, X]_{C_2}\to [C_{2\,+}\wedge S^{n}, c^*_{L/k}(X)^{\wedge}_{\eta}]_{k}$
are isomorphisms for all $n$. This is a localizing subcategory and by \aref{prop:csphere} it contains $(\SS_{C_2})^{\wedge}_{\eta}$ and we argue below that 
$C_{2\,+}\wedge (\SS_{C_2})^{\wedge}_{\eta}$ is in $\mcal{C}$ as well. This implies that $\mcal{C}=(\SH_{C_2})^{\wedge}_{\eta}$, as this is the smallest localizing subcategory containing 
$\{(\SS_{C_{2}})^{\wedge}_{\eta}, C_{2\,+}\wedge (\SS_{C_{2}})^{\wedge}_{\eta}\}$.

Now we show that
 $C_{2\,+}\wedge(\SS_{C_2})^{\wedge}_{\eta}$ is also in $\mcal{C}$. 
Since $c^{*}_{L/k}$ is strong symmetric monoidal and $C_{2\,+}$ is dualizable, \cite[Proposition 3.12]{HuFauskMay:isos} implies that for any $C_{2}$-spectrum $X$, 
the natural map
$c^*_{L/k}(F(C_{2\,+},X))\to F(\spec(L)_{+}, c^*_{L/k}(X))$ 
is an isomorphism in $\SH_{k}$, where $F(-,-)$ denotes the function spectrum in the corresponding homotopy category.
Now $C_{2\,+}$ is self dual, \emph{i.e.}~there is an isomorphism $C_{2\,+}\iso D(C_{2+})$ in $\SH_{C_{2}}$ where $D(-)= F(-,\SS_{C_2})$
denotes the Spanier-Whitehead dual. As with any dualizable object, there is a natural isomorphism $\nu:D(C_{2\,+})\wedge X \iso F(C_{2\,+}, X)$. Combining these isomorphisms yields the isomorphism
$\omega:C_{2\,+}\wedge X \iso F(C_{2\,+}, X)$ in $\SH_{C_2}$, which is a simple case of the Wirthm\"{u}ller isomorphism, and $c^*_{L/k}(\omega)$ induces an isomorphism $\spec(L)_{+}\wedge c^*_{L/k}(X)^{\wedge}_{\eta} \iso F(\spec(L)_+,c^*_{L/k}(X)^{\wedge}_{\eta})$ in $\SH_{k}$.
This isomorphism together with \aref{prop:csphere} now 
implies that the maps
\begin{enumerate}
 \item[(i)] $(c^*_{L/k})^{\wedge}_{\eta}:[S^{n}, C_{2\,+}\wedge (\SS_{C_2})^{\wedge}_{\eta}]_{C_2}\to 
[S^{n}, \spec(L)_{+}\wedge c^*_{L/k}((\SS_{C_2})^{\wedge}_{\eta})^{\wedge}_{\eta}]_{k}$, and 
 \item[(ii)] $(c^*_{L/k})^{\wedge}_{\eta}:[C_{2\,+}\wedge S^{n}, C_{2\,+}\wedge (\SS_{C_{2}})^{\wedge}_{\eta}]_{C_2}\to [\spec(L)_+\wedge S^{n}, \spec(L)_{+}\wedge c^*_{L/k}((\SS_{C_2})^{\wedge}_{\eta})^{\wedge}_{\eta}]_{k}$
\end{enumerate}
 are isomorphisms for any $n\in \Z$.

Now, for any $\eta$-complete $C_2$-spectrum $X$, let $\mathcal{L}_X$ denote the
full subcategory of $\eta$-complete $C_2$-spectra $Y$ such that 
$[S^n \wedge Y,X]_{C_2} \to 
[ S^n\wedge c^*_{L/k}(Y)^{\wedge}_{\eta}, c^*_{L/k}( X)^{\wedge}_{\eta}]_{k}$ is an isomorphism for all
$n\in \ZZ$.  It is clear that $\mathcal{L}_X$ is a localizing
subcategory of $\SH_{C_2}$.  We have seen that $\mcal{L}_{X}$ contains both $(\SS_{C_2})^{\wedge}_{\eta}$ and $C_{2\,+}\wedge (\SS_{C_{2}})^{\wedge}_{\eta}$. Therefore $\mcal{L}_{X}=(\SH_{C_2})^{\wedge}_{\eta}$.
Since $X$ was arbitrary, we have proved that $(c^*_{L/k})^{\wedge}_{\eta}$ is full and faithful.

\end{proof}

Indepedent of whether $k$ has motivic Serre finiteness, the argument in the previous theorem yields the embedding theorem for the $(p,\eta)$-complete homotopy categories. 

\begin{theorem}\label{thm:textmain2}
Let $k$ be a real closed field and let $L=k[i]$ be its algebraic closure. 
Then for any prime $p$
$$
c_{L/k}^*:(\SH_{C_2})^{\wedge}_{p,\eta}\to (\SH_k)^{\wedge}_{p,\eta}
$$
is a full and faithful embedding. For $p=2$, $c_{L/k}^*:(\SH_{C_2})^{\wedge}_{2}\to (\SH_k)^{\wedge}_{2}$ is full and faithful.
\end{theorem}

As mentioned in the introduction, our main theorem has the following
corollary on Picard-graded stable homotopy groups.

\begin{corollary} \label{cor:groups}
Suppose $k$ is real closed and $L=k[i]$ and let $S^L$ denote the
unreduced suspension of $\spec(L)$.  Then for all $m,n\in \ZZ$ and any $(p,\eta)$-complete $C_2$-spectrum $X$, the functor
$c_{L/k}^*$ induces an isomorphism of Picard-graded stable homotopy
groups
\[
  \pi_{m+n\sigma} (X)\xrightarrow{\iso} \pi_{m+nL}(c_{L/k}^*(X)^{\wedge}_{\eta}).
\]
If $k$ has motivic Serre finiteness, then it is an isomorphism for any $C_2$-spectrum $X$. In particular, in this case
\[
  \pi_{m+n\sigma}((\SS_{C_2})^{\wedge}_{\eta}))\xrightarrow{\iso} \pi_{m+nL}( (\SS_k)^{\wedge}_{\eta}).
\]
\end{corollary}

We can also deduce a 2-complete version of Morel's conjecture on $\pi_1 \SS_k$ for $k$ a real
closed field.  Recall that for a general field $k$, Morel's conjecture
states that there is a short exact sequence
\[
  0\to K^M_2(k)/24\to \pi_1 \SS_k\to K^M_1(k)/2\oplus \ZZ/2\to 0
\]
in which the map $\pi_1 \SS_k\to K^M_1(k)/2\oplus \ZZ/2$ is induced by
the unit map $\SS_k\to \mathrm{KO}$ to Hermitian $K$-theory and
$K^M_2(k)/24\to \pi_1\SS_k$ takes symbols $[a,b]$ to $[a,b]\nu$, $\nu$
the motivic quaternionic Hopf map.

\begin{corollary} \label{cor:Morel}
If $k$ is real closed, then $\pi_1(\SS_k)^{\wedge}_{2}$ sits in the short exact
sequence
\[
    0\to K^M_2(k)/8\to \pi_1 (\SS_k)^{\wedge}_{2}\to K^M_1(k)/2\oplus \ZZ/2\to 0.
\]
\end{corollary}
\begin{proof}
By \cite{AI:involutions}, we have $\pi_1\SS_{C_2} = (\ZZ/2)^3$ with basis
$\eta_s,[C_2/e]\eta_s,e^2\nu_{C_2}$ where $[C_2/e]$ is the class of $C_2/e$ in $A(C_2)$, $e$ is represented by the
canonical map $S^0\to S^\sigma$, and $\nu_{C_2}$ is the
$C_2$-equivariant quaternionic Hopf
map.  By \aref{cor:groups}, there is an abstract isomorphism
$\pi_1(\SS_k)^{\wedge}_{2} \iso (\ZZ/2)^3$.  (Recall that $(2,\eta)$-completion is the same as $2$-completion when the 2-primary cohomological dimension of $k[i]$ is finite.)  By \cite[Lemma
5.12]{OrmsbyOstvaer:pi1}, the map $\pi_1(\SS_k)^{\wedge}_2\to
K^M_1(k)/2\oplus \ZZ/2$ is surjective, taking $\langle u\rangle
\eta_s$ to $([u],1)$ (where $\langle u\rangle$ represents the quadratic form $uX^2$ in $GW(k)$).  It follows that $\eta_s$ and $\langle
-1\rangle\eta_s$ are linearly independent.  The $C_2$-Betti
realization of $\rho^2\nu$ is $e^2\nu_{C_2}\ne 0$, and $\nu=0\in
\pi_{1+2\alpha}\mathrm{KO} = 0$, so $\rho^2\nu$ is nonzero and
linearly independent of $\eta_s,\langle -1\rangle\eta_s$.  The
corollary follows.
\end{proof}

\begin{remark}
If $k$ is real closed, the map $\pi_1\SS_{C_2}\to
\pi_1\SS_k$ is given by
\[
  \eta_s\mapsto \eta_s,\quad [C_2/e]\eta_s\mapsto \langle 1,-1\rangle
  \eta_s,\quad e^2\nu_{C_2}\mapsto \rho^2\nu.
\]
\end{remark}

Finally we note that an equivariant embedding theorem implies a nonequivariant embedding theorem. 

\begin{corollary} \label{cor:rel}
Let $L/k$ be a finite Galois extension with Galois group $G$.  If the functor
$c_{L/k}^*:\SH_{G}\to \SH_{k}$ is full and faithful, then the constant presheaf functor
$c^*_{L/L}:\SH\to \SH_{L}$ is full and faithful as well.  
\end{corollary}
\begin{proof}
Assume that $c_{L/k}^*$ is full and faithful and consider the commutative diagram
\[
\xymatrix{
  [G_{+}\wedge S^n,X]_{G}\ar[d]_\iso\ar[r]^-{c_{L/k}^*}_-\iso
  &[c_{L/k}^*(G_{+}\wedge S^n),c^*_{L/k}X]_{k}\ar[d]^\iso\\
  [S^n,\res X]_{e}\ar[r]_-{c_{L/L}^*} &[c^* S^n,c^*\res X]_{L},
}
\]
obtained using \aref{prop:stablecompats}.
 The vertical
arrows are isomorphisms,
and the top horizontal arrow is an isomorphism by assumption.  Thus
the bottom horizontal arrow is an isomorphism as well.  Since every
spectrum is the restriction $\res X$ of  some $G$-spectrum $X$, we can use
a density argument as in the proof of \aref{thm:textmain} to conclude that $c_{L/L}^*$ is full and
faithful. 
  
\end{proof}

\section{The trace homomorphism and necessary conditions for
  full-faithfulness} \label{sec:classical}
In this section, we discuss the
possibility of  $c^*_{L/k}$ being full and faithful for more general Galois extensions $L/k$. 
As noted in the introduction, presence of torsion in the Grothendieck-Witt group is the first obvious obstruction to  an isomorphism on $\pi_{0}$ and therefore to $c^*_{L/k}$ inducing a full and faithful embedding. However, there are many fields whose Grothendieck-Witt group is torsion-free and we are able to place strong restrictions on which fields $k$ and $L$ 
can have the property that $c^*_{L/k}$ induces an isomorphism on $\pi_{0}$.

Recall that the classical map $h_{L/k}:A(G)\to GW(k)$ mentioned in the introduction is the unique ring homomorphism with the property that $G/H\in A(G)$ is mapped to 
 $\tr_{L^H/k}(\langle 1\rangle)$.  (See \cite[\S 4]{BeaulieuPalfrey} 
for the basic properties of $h_{L/k}$.) 
The functor $c^*_{L/k}:\SH_{G}\to \SH_{k}$ also induces a map $c^*_{L/k}:A(G)\to GW(k)$.
The following is essentially a rephrasing of M. Hoyois's \cite{Hoyois:trace} computation of the motivic Euler characteristic of a separable field extension.

\begin{proposition}\label{prop:oldequalsnew}
  The maps $c_{L/k}^*:A(G)\to GW(k)$ and $h_{L/k}$  are equal. 
\end{proposition}
\begin{proof}  
 The identification $A(G)\iso \End_{\SH_{G}}(\SS_{G})$ is given by sending a finite $G$-set $M$ to its Euler characteristic $\chi(M)$ (for a recollection of Euler characteristics and their properties see, \emph{e.g.}, \cite{May:GeneralPic}). The functor $c^*_{L/k}$ is strong symmetric monoidal and so we have that 
 $c^*_{L/k}\chi(G/H) = 
 \chi(c^*_{L/k}(G/H))=\chi(\spec(L^{H}))$ in $\End_{\SH_{k}}(\SS_{k})$.  But by \cite[Theorem 7]{Hoyois:trace}, under the identification $\End_{\SH_{k}}(\SS_{k})\iso GW(k)$, we have  $\chi(\spec(L^H)) = \tr_{L^H/k}(\langle 1 \rangle)$.
 \end{proof}

A field $k$ is \emph{pythagorean} if and only if sums of
squares in $k$ are squares in $k$.  Since $A(G)$ is always torsion
free as an abelian group, the importance of pythagorean
fields in our context is illustrated by the following lemma.

\begin{lemma} \label{lem:pyth}
The abelian group underlying $GW(k)$ is torsion free if and only if the field
$k$ is pythagorean.  If $k$ is pythagorean with finitely many orderings, then the free rank of
$GW(k)$ is $1+x(k)$ where $x(k)$ denotes the number of orderings of $k$.
\end{lemma}
\begin{proof}
This is a standard enhancement of \cite[Theorem VIII.4.1 \& Corollary VIII.6.15]{Lam:intro} from the Witt ring to Grothendieck-Witt ring case.
\end{proof}

We will also need the following lemma in order to analyze $h_{L/k}$.

\begin{lemma} \label{lem:x}
If $k$ is pythagorean and $[k^\times \smash : \, (k^\times)^2] = 2^n$, then
\[
  n\le x(k)\le 2^{n-1}.
\]
\end{lemma}
\begin{proof}
This is a specialization of \cite[Exercise VIII.16]{Lam:intro}.
\end{proof}

Recall that $k$ is \textit{euclidean} if $-1$ is not a sum of squares in $k$ and $[k^\times \smash : \, (k^\times)^2]=2$.

\begin{theorem} \label{thm:necessary}
The map $h_{L/k}$ is an isomorphism if and only if either $k$ is quadratically
closed and $L=k$, or $k$ is euclidean and $L = k[i]$.
\end{theorem}
\begin{proof}
If $L/k$ is of one of the prescribed forms, then it is elementary that
$h_{L/k}$ is an isomorphism.

If $h_{L/k}$ is an isomorphism, then $GW(k)$ must be torsion free, in
which case \aref{lem:pyth} implies that $k$ is pythagorean.  If $k$ is
pythagorean and nonreal (\emph{i.e.}, $-1$ is a sum of squares in
$k$), then $k$ is quadratically closed and $GW(k) \cong \ZZ$.  Thus
$A(G)$ has rank $1$ and therefore $G = \{e\}$ and $L = k$.

Now assume that $k$ is pythagorean and formally real (so $-1$ is not a sum
of squares in $k$).  By the construction in \cite[\S
4]{BeaulieuPalfrey}, we know that $h$ factors through the group
completion of the monoid of $k$-quadratic forms $q$ such that
$q_L\cong n\langle 1\rangle$ for some natural number $n$; call this
group $GW^{\ZZ}_L(k)$.  Since $h$
is an isomorphism, $GW^\ZZ_L(k) = GW(k)$, whence $\langle a \rangle_L
= \langle 1\rangle$ in $GW(L)$ for all $a\in k^\times$.  It follows
that $k$ is quadratically closed in $L$.

Choose a basis $\{x_1,x_2,\ldots\}$ of $k^\times/(k^\times)^2$ and let
$E = k(\sqrt{x_1},\sqrt{x_2},\ldots)$.  We
have just proven that 
$E/k$ is a subextension of $L/k$, whence
$G$ surjects onto $\Gal(E/k)$.  Since $G$ is finite, $k$ must have
finitely many square classes and $\Gal(E/k)\cong C_2^n$.  Recall that the rank of $A(G)$ is the
number of conjugacy classes of subgroups of $G$.  We deduce that $\rk
A(G)\ge \rk A(C_2^n)$.  Just
counting the subgroups of $C_2^n$ of order $1$ or $2$, we find that
$\rk A(C_2^n)\ge 2^n$.  Since $2^n>1+2^{n-1}$ for $n>2$, \aref{lem:x}
implies that $n=0$ or $1$.  Since $k$ is formally real, we can exclude
the case $n=0$, whence $k$ is formally real pythagorean with
$[k^\times \smash : \, (k^\times)^2] = 2$, \emph{i.e.}, $k$ is euclidean.  In this
case $GW(k)$ has rank $2$, so $L/k$ is a quadratic extension.  Since
$k$ is quadratically closed in $L$, $L = k[i]$, concluding the proof.
\end{proof}

\begin{corollary} \label{cor:necessary}
If $c_{L/k}^*$ is full and faithful, then $k$ is of the form
described in \aref{thm:necessary}.
\end{corollary}

\begin{remark} \label{rmk:pi1}
Algebraically closed and real closed fields are special examples of
quadratically closed and euclidean fields, but there are many other examples of these kinds of fields.
For instance, the field of real
constructible numbers $\widetilde{\QQ}\cap \RR$ (where
$\widetilde{\QQ}$ is the quadratic closure of $\QQ$) is euclidean but
not real closed.  

The necessary conditions which we deduced in the previous result were obtained only by analyzing the zeroth homotopy group of the sphere spectrum.
The authors expect that
torsion phenomena in the higher homotopy groups of $\SS_k$ will
preclude $c_{L/k}^*$ from being full and faithful unless $k$ is
algebraically or real closed.
\end{remark}

\begin{conjecture}
 Let $L$ be a field of characteristic zero. The functor $c_{L/L}^{*}:\SH\to \SH_{L}$ is full and faithful if and only if $L$ is algebraically closed. 
\end{conjecture}
By Levine's theorem \cite{Levine:comparison} the ``if'' portion of this conjecture is valid. Observe that the validity of this conjecture together with \aref{cor:rel}, would imply that $c_{L/k}:\SH_{G}\to \SH_{k}$ is full and faithful if and only if $k=L$ is algebraically closed or $k$ is real closed and $L=k[i]$.
It is also interesting to ask what happens in positive characteristic.  

\section{Comparison functors}\label{sec:constructions}
In this section we construct and analyze the various comparison functors between stable homotopy categories used in our arguments. To avoid potential confusion concerning notation, we point out that a functor on homotopy categories written as the derived functor $\L F$ (or $\R F)$ of some functor on model categories in this section would be written simply $F$ in previous sections.

\subsection{Motivic model structures}\label{sub:models}
Given a base scheme $S$, the category $\MS(S)$ of based motivic spaces is the category of based simplicial presheaves on $\sm/S$. There are many different options for a motivic model structure on $\MS(S)$. We will use  the so-called closed flasque motivic model structure  introduced in \cite{PPR:KGL}. We recall the basic definitions below and refer to \emph{loc.~cit.}~for full details.   
 The main advantages of this model structure for the present work are that in this model structure all of the standard motivic spheres are cofibrant and all of the various change of base functors as well as the (equivariant) Betti realizations are Quillen functors.

 The closed flasque motivic model structure is a Bousfield localization of the global closed flasque model structure.
The weak equivalences of the global closed flasque model structure are the schemewise weak equivalences of motivic spaces. A global closed flasque fibration is a map which has  the right lifting property with respect to the set $J^{gf}$ defined below. A global closed flasque cofibration is then defined by the appropriate lifting property.
This model structure has sets of generating cofibrations $I^{gf}$ and generating acyclic cofibrations $J^{gf}$ as follows.
Let $\mcal{Z}= \{Z_{i}\to X\}$ be a finite (possibly empty) collection of
closed immersions
in $\sm/S$. Write $\cup \mcal{Z}$ for the 
  categorical union (i.e. union as presheaves) of the $Z_i$ and write $f:\cup \mcal{Z}\to X$ for the induced map. Given two maps $\alpha$ and $\beta$ write $\alpha\square\beta$ for their pushout product.
    \begin{enumerate}
  \item The set $I^{gf}$  consists of all maps of the form $f_+\square \;g_+$ where $f:\cup\mcal{Z}\to X$ is as above and $g:\partial\Delta^{n}\to \Delta^n$ is a generating cofibration of simplicial sets.
   \item  The set  $J^{gf}$ consists of all morphisms of the form $f_+\square\; h_+$, where $f:\cup\mcal{Z}\to X$ is as above and $h:\Lambda^{n}_{j}\to \Delta^{n}$ is a generating acyclic cofibration.
  \end{enumerate}
The global closed flasque model structure is a proper, cellular, simplicial model structure.
Write $\MS(S)_{gf}$ for the category of motivic spaces equipped with the global model structure.
Let 
$$
\xymatrix{
B \ar[r]\ar[d] & Y \ar[d]^{p}\\
A \ar[r]^{i} & X.
}
$$
be a distinguished Nisnevich square, i.e.~ $p$ is an \'etale map of smooth schemes, $i$ is an open immersion, and $p^{-1}(X\setminus A)_{red}\to (X\setminus A)_{red}$ is an isomorphism. Write $\mcal{Q}=\mcal{Q}(i,p)$ for this distinguished square and write $P_{\mcal{Q}}$ for the homotopy pushout in $\MS(S)_{gf}$ of $A$ and $Y$ along $B$, and write $P_{\mcal{Q}}\to X$ for the resulting map.
The \textit{motivic closed flasque model structure} is the left Bousfield localization  of the global model structure at the set of maps 
$$
S = \{ P_{\mcal{Q}}\to X \} \cup \{ W\times\A^{1}\to W\}
$$
where $X$, $W$ range over all smooth $S$-schemes and $\mcal{Q}$ ranges over all distinguished squares.

\subsection{Stable model structures}\label{sub:stablemodels}
We rely on \cite{Hovey:spectra} as needed to equip various categories of spectra (and bispectra) with stable model structures. 
Recall that if $\mcal{C}$ is a left proper cellular symmetric monoidal
model category whose generating cofibrations have cofibrant domain and
$K$ is a cofibrant object of $\mcal{C}$, then   
 Hovey equips the category 
$\sspt_{K}(\mcal{C})$ of symmetric $K$-spectra with a stable model structure and it is again a left proper cellular symmetric monoidal model category \cite{Hovey:spectra}. 
Note that $\MS(S)$ satisfies these assumptions and moreover the motivic spheres $\P^1$ (based at $\infty$), $\A^{1}/\A^{1}\setminus\{0\}$, and $S^{\alpha}:=\A^{1}\setminus\{0\}$ (based at $1$) are all closed flasque cofibrant.

Let $J$ be a closed flasque cofibrant motivic space over $S$. We will simply write 
$\sspt_{J}(S):= \sspt_{J}(\MS(S))$ 
for the category of motivic $J$-spectra. If $J'$ is another closed flasque cofibrant motivic space we write  
$\sspt_{J,J'}(F):= \sspt_{J'}(\sspt_{J}(F))$ for the category of motivic $(J,J')$-bispectra. 
As shown in \cite{PPR:KGL} there is a monoidal Quillen equivalence between  $\sspt_{\P^{1}}(S)$ and
Jardine's model category of motivic symmetric $\P^1$-spectra \cite{Jardine:motivicspectra}.

 In \cite{Hovey:spectra}, functoriality of the model categories of
 symmetric spectra is discussed when $\mcal{C}$ is fixed (\emph{e.g.}~ changing the suspension object $K$ in $\mcal{C}$ or varying the $\mcal{C}$-model category). We will need slightly more general functoriality, which we record before continuing with the construction of the comparison functors of interest to this paper.

Suppose that $\mcal{D}$ is another model category satisfying the same hypothesis as $\mcal{C}$ and $K'$ is a cofibrant object of $\mcal{D}$. Further suppose that we are given the following: 
\begin{enumerate}
 \item a Quillen adjoint pair $\Phi:\mcal{C}\rightleftarrows\mcal{D}:\Psi$, and
 \item a  natural isomorphism $\tau:\Phi(-)\otimes K'\xrightarrow{\iso} \Phi(-\otimes K) $ such that
the iterated isomorphisms $\tau^{p}:\Phi(X)\otimes (K')^{\otimes p}\iso
\Phi(X\otimes K^{\otimes p})$ are $\Sigma_{p}$-equivariant, where the actions are the obvious ones given by permuting the respective factors of $K$ and $K'$. 
\end{enumerate}
 
As seen in the next lemma, $(\Phi,\Psi)$ prolong to a Quillen pair $(\Sp(\Phi), \Sp(\Psi))$ of stable model categories of symmetric spectra. In this situation, we usually write $\Phi$ and $\Psi$ instead of $\Sp(\Phi)$ and $\Sp(\Psi)$  for the prolongations. 

\begin{lemma}\label{lem:functoriality}
With notations and assumptions as above, the pair
$(\Phi, \Psi)$
prolongs to a Quillen adjoint pair on stable model structures
$$
\Sp(\Phi):\sspt_{K}(\mcal{C})\rightleftarrows \sspt_{K'}(\mcal{D}):\Sp(\Psi).
$$
 If $\Phi$ is strong symmetric monoidal then so is $\Sp(\Phi)$.
\end{lemma}
\begin{proof}
Define $\Sp(\Phi)(D)$ by $\Sp(\Phi)(D)_{n} := \Phi(D_{n})$ with structure maps  
$$
\Sp(\Phi)(D)_{n}=\Phi(D_n)\otimes K' \iso \Phi(D_{n}\otimes K) \to \Phi(D_{n+1})=\Sp(\Phi)(D)_{n+1}.
$$ 

The equivariance assumption on $\tau$ implies that the iterations of the structure map $\Sp(\Phi)(D)_{n}\otimes (K')^{\otimes p}\to \Sp(\Phi)(D)_{n+p}$ are $\Sigma_{n}\times \Sigma_{p}$-equivariant and so $\Phi(D)$ is a symmetric $K'$-spectrum. Define $\Sp(\Phi)$ on morphisms in the obvious way.

Note that $\tau$ determines the natural isomorphism $\rho:\Psi\Omega_{K'}(-) \xrightarrow{\iso} \Omega_{K}\Psi(-)$ and the iterations  $\rho^p$ are $\Sigma_p$-equivariant. Now define $\Sp(\Psi)(E)$ by setting $\Sp(\Psi)(E)_{n} := \Psi(E_n)$. The structure maps are defined as the adjoints of 
$$
\Sp(\Psi)(E)_{n}=\Psi(E_{n}) \to \Psi(\Omega_{K'}E_{n+1}) \iso \Omega_{K}\Psi(E_{n+1})=\Omega_{K}\Sp(\Psi)(E)_{n+1}.
$$
The equivariance of $\rho$ implies that this is a symmetric $K$-spectrum. Define $\Sp(\Psi)$ on morphisms in the obvious way.

It is straightforward to verify that $\Sp(\Phi)$ and $\Sp(\Psi)$ are adjoint. The functor $\Sp(\Psi)$ preserves level equivalences and level fibrations. This implies $\Sp(\Phi)$ preserves stable cofibrations and $\Sp(\Psi)$ preserves fibrations between fibrant objects in the stable model structure. It follows from 
\cite[Lemma A.2]{dugger:simp} that $(\Sp(\Phi),\Sp(\Psi))$ is a Quillen adjoint pair on the stable model structures.

It is immediate that $\Sp(\Phi)$ is symmetric monoidal whenever $\Phi$ is.
\end{proof}

\subsection{Galois correspondence}\label{sub:galois}
Let $L/k$ be a finite Galois extension with Galois group $G$.
Define the functor
\begin{equation}\label{eqn:galoiscorrespondence}
c_{L/k}:\Or_{G}\to \Sm/k,
\end{equation}
by $c_{L/k}(G/H) = \spec(L^{H})$ on objects and on maps as follows. First recall that 
$$
\Hom_{\Or_{G}}(G/H, G/H') = \{gH' \,|\, g^{-1}H g \subseteq H'\}.
$$
A straightforward check shows that if $gH'$ is such a coset then the corresponding field automorphism $g:L\to L$ restricts to a map of fields  $g:L^{H'}\to L^{H}$ which  depends only on the coset $gH'$. This defines the desired map $c_{L/k}(G/H)\to c_{L/k}(G/H')$.

The category of $G$-simplicial sets is equivalent to the category of presheaves of simplicial sets on $\Or_{G}$: the presheaf corresponding to $A$ is given by $G/H\mapsto A^{H}$. We thus obtain an adjoint pair of functors
\begin{equation}\label{eqn:unstablepair}
c^*_{L/k}:G\ssetb \rightleftarrows\MS(k):(c_{L/k})_{*} .
\end{equation}

\begin{remark}\label{rem:useful}
For a $G$-simplicial set $A$, the corresponding motivic space $c^*_{L/k}(A)$ isn't in general constant but its possible values are limited to the various fixed points $A^{H}$ for subgroups $H\subseteq G$. To see this, it suffices to consider the case of a $G$-set. Every $G$-set is the disjoint union of orbits and we write this decomposition as 
$A = \coprod_{\Or_{G}}\coprod_{A\langle{G/H}\rangle} G/H$.
 Then $c^*_{L/k}A$ is the motivic space defined by
$$
 (c_{L/k}^*A)(X):= \coprod_{\Or_{G}}\coprod_{A\langle {G/H}\rangle}
 \Hom_{k}(X, \spec(L^{H})).
$$
Note that if $X$ is connected, then $\Hom_{k}(X, \spec(L^{H}))$ is either empty or is a set with 
$|G/H|$ elements and so $c_{L/k}^{*}(A)(X) = A^{H}$ for an appropriate subgroup $H\subseteq G$.
\end{remark}

\begin{lemma}\label{lemma:uff}
 The adjoint pair $c^*_{L/k}:G\ssetb \rightleftarrows\MS(k):(c_{L/k})_{*} $ is a Quillen adjoint pair. Moreover, the induced map on homotopy categories $\L c^*_{L/k}: \HH_{\bullet, G} \to \HH_{\bullet, k}$ is full and faithful.
\end{lemma}
\begin{proof}
 Note that under the identification $\spre_{\bullet}(\Or_{G})=G\ssetb$, the projective model structure on simplicial presheaves corresponds to the usual model structure on based $G$-simplicial sets.   The functor $(c_{L/k})_*$  preserves global weak equivalences and global fibrations and so this pair is a Quillen pair on the global closed flasque model structure. It follows immediately that this is a Quillen pair on the motivic model structure as well.

 Using the description of $c^*_{L/k}$ in the previous remark, one sees the following simple facts about $c^*_{L/k}(A)$. If $A$ is fibrant then $c^*_{L/k}(A)(X)$ is fibrant for any $X$ and $c^*_{L/k}(A)$ is $\A^1$-homotopy invariant. If $U\subseteq X$ is a dense open subscheme, then $c^*_{L/k}(A)(X) = c^*_{L/k}(A)(U)$. It is thus easy to see that $c^*_{L/k}(A)$ satisfies Nisnevich descent. 
 Moreover, for $G$-simplicial sets $A$ and $B$,  we have an equality of simplicial mapping spaces, $\ul{\Hom}_{\MS(k)}(c^*_{L/k}(B), c^*_{L/k}(A)) = \ul{\Hom}_{G\ssetb}(B,A)$.
 Now if $B$ is cofibrant and $A$ is fibrant, then we have
\[
  [S^{n}\wedge c^*_{L/k}(B), c^*_{L/k}(A)]_{k}= 
\pi_{n}\ul{\Hom}_{\MS(k)}(c^*_{L/k}(B), c^*_{L/k}(A))
\]
and $[B,A]_{G}= \pi_{n}\ul{\Hom}_{G\ssetb}(B,A) $ from which the second statement follows.
 \end{proof}

Write $S^{G} = (S^1)^{\wedge G}$ for the $G$-simplicial set consisting of the $|G|$-fold smash product of $S^1$ equipped with the obvious permutation action by $G$. Note also that this is the simplicial representation sphere associated to the regular representation of $G$. The stable model structure on $\sspt_{S^{G}}(G):=\sspt_{S^{G}}(G\ssetb)$ obtained from \cite{Hovey:spectra}  agrees with that constructed in \cite{Mandell:Gsymmetric}. In turn, as shown in \emph{loc.~cit.}, the associated homotopy category is tensor triangulated equivalent to the genuine $G$-equivariant homotopy category as constructed in \cite{LMS}.

To simplify notation below, we sometimes denote the motivic space $c^*_{L/k}(S^{G})$ by $S^{G}$.
Consider the category $\sspt_{S^{G},\P^1}(k)$ of
motivic $(c^*_{L/k}(S^{G}),\P^1_k)$-bispectra. This is a  model for the stable motivic homotopy category $\SH_{k}$. Indeed, by \cite[Theorem 3.5]{Hu:base} the motivic space 
$c^*_{L/k}(S^{G})$ is invertible in $\SH_{k}$. In particular, by \cite[Theorem 9.1]{Hovey:spectra}, the suspension spectrum functor 
$$
\Sigma^{\infty}_{S^G}:\sspt_{\P^1}(k)\to \sspt_{\P^1,S^G}(k)
$$
is a left Quillen equivalence and induces a tensor triangulated equivalence on the associated stable homotopy categories.

By \aref{lem:functoriality}, the Quillen adjoint pair (\ref{eqn:unstablepair})
induces a Quillen pair $\sspt_{S^{G}}(G) \rightleftarrows \sspt_{S^{G}}(k)$. Combined with the suspension spectrum functor, we have the composite Quillen adjunction
\begin{equation*}
\sspt_{S^G}(G)\rightleftarrows \sspt_{S^{G}}(k)
\rightleftarrows  \sspt_{S^{G},\P^1}(k).
\end{equation*}
We have thus obtained the desired stabilization of $c^*_{L/k}$.
\begin{theorem}\label{thm:construction}
The Galois correspondence (\ref{eqn:galoiscorrespondence}) induces 
an adjoint pair
$$
\L c^*_{L/k}:\SH_{G}\rightleftarrows \SH_{k} :\R(c_{L/k})_*
$$
of triangulated stable homotopy categories. The left adjoint is strong symmetric monoidal. 
\end{theorem}

\subsection{Equivariant Betti realization}\label{sub:betti}
An unstable $C_2$-equivariant Betti realization functor is constructed for the motivic homotopy category over fields admitting a real embedding in \cite{MV:A1}, see also \cite{DI:hyp}. It is well known that this construction stabilizes to yield a $C_2$-equivariant Betti realization functor. Following the construction of \cite{PPR:KGL} in the complex case, we record here the construction of the stable equivariant Betti realization as a Quillen functor.

Write $(-)^{an}:\Sm/\R\to C_{2}\btop$ for the functor given by 
$X\mapsto X(\C)^{an}_+$, where $X(\C)$ is equipped with the involution given by conjugation. It extends to an adjoint pair 
$$
\RRe^{C_2}_{B}:\MS(\R)\rightleftarrows C_2\btop:\sing^{C_2}_{B}.
$$ 
The left adjoint $\RRe^{C_2}_{B}$ is defined by the usual left Kan extension formula and the  right adjoint $\sing^{C_2}_{B}$ is defined by $\sing^{C_2}_{B}(K)(X) = \ul{\Hom}_{C_2\btop}(X(\C)_+, K)$.

\begin{proposition} 
 The adjoint pair
$
 \RRe^{C_2}_{B}:\MS(\R)\rightleftarrows C_2\btop:\sing^{C_2}_{B}
 $
 is a Quillen adjoint pair. Moreover $\RRe^{C_2}_{B}$ is strong symmetric monoidal.
\end{proposition}
\begin{proof}
First we show that this is a Quillen pair on global closed flasque model structures. For this we check that $\RRe^{C_{2}}_{B}$ sends generating closed cofibrations to cofibrations in $C_{2}\btop$ and sends generating global  trivial closed fibrations to trivial cofibrations.
Note that $\RRe^{C_{2}}_{B}$ preserves pushout products. It thus suffices to show  that $\RRe^{C_2}_{B}(\cup\mcal{Z}_+)\to \RRe^{C_{2}}_{B}(X_+)$ is a cofibration for any finite collection $\mcal{Z}= \{Z_{i}\hookrightarrow X\}$ of closed immersions in $\Sm/\R$. 

Note that $\RRe^{C_{2}}_{B}(\cup \mcal{Z})$  is the coequalizer of $\coprod Z_{i}(\C)\times_{X(\C)} Z_{j}(\C) \rightrightarrows \coprod Z_{i}(\C)$ in in $C_2\btop$. One may equivariantly triangulate $X(\C)$ such that each 
$Z_{i}(\C)$ is an equivariant subcomplex and $Z_{i}(\C)\times_{X(\C)}Z_{j}(\C)$ is an equivariant subcomplex for each $j$, see, \emph{e.g.}, \cite{Illman}. 
It follows that $\RRe^{C_2}_{B}(\cup \mcal{Z})\to X(\C)$ is the inclusion of an equivariant subcomplex. In particular, it is an equivariant cofibration. It follows that $\RRe^{C_{2}}_{B}$ is a left Quillen functor on the global closed flasque model structure.

Note that $\RRe^{C_2}_{B}$ sends a distinguished Nisnevich square to
an equivariant homotopy pushout square, see, \emph{e.g.}, \cite{DI:hyp}. Also $\RRe^{C_2}_{B}(X\times \A^{1}) \to \RRe^{C_{2}}_{B}(X)$ is an equivariant homotopy equivalence. It follows that the adjoint pair of the proposition 
induces a Quillen pair in the closed flasque motivic structure as well.
\end{proof}

Recall that we write $S^{\sigma}$ for the sign representation sphere. 
\begin{proposition}
 The above adjoint pair extends to a Quillen adjoint pair 
 $$
 \RRe^{C_2}_{B}:\sspt_{\P^1}(\R) \rightleftarrows \sspt_{S^{1+\sigma}}(C_{2}):\sing^{C_2}_{B}
 $$
 on stable model categories. Moreover $\RRe_{B}^{C_{2}}$ is strong symmetric monoidal.
\end{proposition}
\begin{proof}
 This follows immediately from \aref{lem:functoriality}, noting that $\RRe^{C_2}_{B}(\P^1) = S^{1+\sigma}$.
\end{proof}

Now if $k$ is a field and $\phi:k\hookrightarrow \R$ is a real embedding then the associated $C_2$-equivariant Betti realization $\RRe^{C_{2}}_{B,\phi}$ is defined to be the composite
$$
\RRe^{C_{2}}_{B,\phi}:= \phi^*\circ\RRe^{C_{2}}_{B}:\SH_{k}\to \SH_{\R}\to \SH_{C_{2}}.
$$

\subsection{Comparing change of group and change of base functors}\label{sub:change}
It is useful to know  that the comparison functors between equivariant and motivic homotopy theory suitably intertwine the standard change of group and change of base functors. We fix as above a Galois extension $L/k$ with Galois group $G$.
Let $H\subseteq G$ be a subgroup and write $K = L^{H}$ for the corresponding fixed subfield. We denote the corresponding map of schemes by $p:\spec(K)\to \spec(k)$. As with any map of schemes we have an induced adjoint pair of functors of motivic spaces 
$p^*:\MS(k)\rightleftarrows \MS(K):p_*$. Since $p$ is smooth, the functor $p^*$ has as well a left adjoint $p_{\#}$, induced by 
the functor $\Sm/K\to \sm/k$ which composes the structure map of a $K$-scheme with $p$.

\begin{lemma}
 The adjoint pairs $(p_{\#},p^*)$ and $(p^*, p_*)$ are Quillen adjoint pairs.
\end{lemma}
\begin{proof}
 That $p^*$ is a left Quillen adjoint on the motivic closed flasque model structure is verified in \cite{PPR:KGL}. 
 Note that $p_{\#}$ preserves generating global closed flasque cofibrations and acyclic cofibrations.
 This is seen by noting that  if $M$ is a simplicial set, then we have that $p_{\#}(X\wedge M_+) = (p_{\#} X) \wedge M_+$ and since $p_{\#}$ preserves colimits, it preserves pushouts and since it also preserves closed inclusions of smooth schemes, the claim follows. This implies that $p_{\#}$ is a left Quillen functor on global closed flasque model structures.
 The functor $p_{\#}$ sends Nisnevich distinguished squares to
  Nisnevich distinguished squares. Furthermore  
  $p_\#(X\times_{K}\A^{1}_{K})\to p_{\#}(X)$ is identified with 
  $p_{\#}(X)\times_{k}\A^{1}_{k}\to p_{\#}(X)$. 
  It follows that $p_{\#}$ is also a left Quillen functor on the closed flasque motivic
  model structure.
 \end{proof}

We have the commutative diagram of categories
$$
\xymatrix{
\Or_{H} \ar[d]_{j}\ar[r]^-{c_{L/K}} & \sm/K \ar[d]^-{p_{\#}} \\
\Or_{G} \ar[r]^-{c_{L/k}} & \sm/k,
}
$$
where $j$ sends the orbit $H/H'$ to the orbit $G/H'$. 
Under the identification  $\spre_{\bullet}(\Or_{G}) = G\ssetb$, the adjoint pair 
$(j^*, j_*)$ is identified with the adjoint pair $(\ind^{G}_{H}, \res^{G}_{H})$ where $\ind^{G}_{H}(X) = G\times_{H} X$ and $\res^{G}_{H}(W)$ is $W$ with $H$-action given by restricting the $G$-action. 
The above square thus induces a commutative diagram of Quillen adjoint
functors (where we omit the labels for the horizontal right adjoints for typographical reasons) 
\begin{equation}\label{eqn:zwei}
\xymatrix{
H\ssetb
\ar@<+.7ex>[r]^-{c_{L/K}^*}\ar@<-.7ex>[d]_{\ind^{G}_{H}} 
& \MS(K) 
\ar@<-.7ex>[d]_{p_{\#}} \ar@<+.7ex>[l]^-{}
\\
G\ssetb 
\ar@<-.7ex>[u]_{\res^{G}_{H}}\ar@<+.7ex>[r]^-{c_{L/k}^*} 
& \MS(k) 
\ar@<-.7ex>[u]_{p^*}\ar@<+.7ex>[l]^-{}.
}
\end{equation}

We write $\HH_{\bullet, G}$ for the homotopy category of based $G$-spaces and $\HH_{\bullet, k}$ for the unstable motivic homotopy category.

\begin{proposition}\label{prop:unstable2}
The diagrams of homotopy categories
$$
\xymatrix{
\HH_{\bullet, H} \ar[r]^{\L c^*_{L/K}} & \HH_{\bullet, K}
\\
 \HH_{\bullet, G} \ar[u]^-{\R \res^{G}_{H}} \ar[r]^{\L c^*_{L/k}} & \HH_{\bullet, k} \ar[u]_-{ \R p^*} 
 }
\;\;\;\;\;\text{and}\;\;\; 
\xymatrix{
 \HH_{\bullet, H} \ar[d]_{\L \ind^{G}_{H}} \ar[r]^{\L c^*_{L/K}} & \HH_{\bullet, K} \ar[d]^{\L p_{\#}} \\
 \HH_{\bullet, G} \ar[r]^{\L c^*_{L/k}} & \HH_{\bullet, k}.
 }
 $$ 
induced by (\ref{eqn:zwei}), commute up to natural isomorphism.
\end{proposition}
\begin{proof}
 The commutativity of the second diagram follows immediately from the fact that (\ref{eqn:zwei}) commutes and that the adjoint pairs there are Quillen pairs. 
    A direct inspection yields the equality of functors
 $p^*c^*_{L/k} = c^*_{L/K}\res^{G}_{H}$. The commutativity of the first diagram follows 
 since $p^*$ and $\res^{G}_{H}$ are also  left Quillen functors and so  $\R p^* \wkeq \L p^*$ and $\R\res^{G}_{H} \wkeq \L \res^{G}_{H}$.
  \end{proof}

As an $H$-simplicial set $S^{G}$ is isomorphic to the $[G:H]$-fold smash product of $S^{H}$. 
This implies that $c^*_{L/K}(S^{G})= c^{*}_{L/k}(S^{H})^{\wedge [G:H]}$. We set $d:=[G:H]$ and write $S^{dH} = (S^{H})^{\wedge d}$ below.

\begin{proposition}\label{prop:stablecompats}
The adjoint pairs (\ref*{eqn:zwei}) induce diagrams of stable homotopy categories
$$
\xymatrix{
\SH_{H} \ar[r]^{\L c^*_{L/K}} & \SH_{K}
\\
 \SH_{G} \ar[u]^-{\R \res^{G}_{H}} \ar[r]^{\L c^*_{L/k}} & \SH_{k} \ar[u]_-{\R p^*} 
 }
\;\;\;\;\;\text{and}\;\;\;
\xymatrix{
 \SH_{H} \ar[d]_{\L \ind^{G}_{H}} \ar[r]^{\L c^*_{L/K}} & \SH_{K} \ar[d]^{\L p_{\#}} \\
 \SH_{G} \ar[r]^{\L c^*_{L/k}} & \SH_{k}
 }
 $$

 which commute up to natural isomorphism.
\end{proposition}
\begin{proof}
We have a diagram of model categories and Quillen adjunctions between them
$$
\xymatrix{
\sspt_{S^G}(H)
\ar@<-.7ex>[d]_{\ind^{G}_{H}}\ar@<+.7ex>[r]^{c^*_{L/K}} 
& \sspt_{S^{dH}}(K)
\ar@<-.7ex>[d]_{p_{\#}}\ar@<+.7ex>[l]^{}
\ar@<+.7ex>[r]^-{\Sigma^{\infty}_{\P^1}} 
& \sspt_{S^{dH},\P^1}(K) \ar@<-.7ex>[d]_{p_{\#}}
\ar@<+.7ex>[l]^-{} \\
\sspt_{S^G}(G)\ar@<+.7ex>[r]^{c^*_{L/k}} 
\ar@<-.7ex>[u]_{\res^{G}_{H}} 
& \sspt_{S^{G}}(k)
\ar@<+.7ex>[r]^-{\Sigma^{\infty}_{\P^1}}
\ar@<+.7ex>[l]^{}  
\ar@<-.7ex>[u]_{p^*}
& \sspt_{S^{G},\P^1}(k)
\ar@<-.7ex>[u]_{p^*}
\ar@<+.7ex>[l]^-{}.
}
$$
This diagram is commutative and the 
derived functors of the left adjoints give the functors in the diagrams. The commutativity of the second diagram follows immediately. 

For the commutativity of the first square, note that the right adjoints $p^*$ and $\res^{G}_{H}$ are also left adjoints and the stabilization of these functors considered as a left adjoint agrees with their stabilization as a right adjoint and these are also left Quillen functors. It follows that $\R p^* = \L p^*$ and $\R\res^{G}_{H} = \L\res^{G}_{H}$. The desired commutativity thus follows from the underived equality 
$\Sigma^{\infty}_{\P^1}c^*_{L/K}\res^{G}_{H} = p^*\Sigma^{\infty}_{\P^1}c^*_{L/k}$.  
\end{proof}

Now suppose that $k$ is formally real and consider the embedding  $p:k\subseteq k[i]$. A real embedding $\phi:k\hookrightarrow \R$ induces
a complex embedding $\psi:k[i]\hookrightarrow \C$ and hence an associated Betti realization $\RRe_{B,\psi} = 
\psi^{*}\RRe_{B}:\SH_{k[i]}\to \SH$. 

\begin{proposition}\label{prop:bettichange}
 With the notations as above we have 
 $$
 \R \res^{C_{2}}_{\{e\}}\L\RRe_{B,\phi}^{C_{2}} = \L\RRe_{B,\psi}\R p^* \;\;\;\mathrm{and}\;\;\; \L\RRe_{B,\phi}^{C_2}\L p_{\#} = \L\ind^{C_{2}}_{\{e\}}\L\RRe_{B,\psi}.
 $$
\end{proposition}
\begin{proof}
 This is a straightforward consequence of the definitions and constructions, as in the previous proposition. 
\end{proof}

\subsection{Betti realization and motivic cohomology}\label{sub:betti2}
We now turn our attention to the equivariant Betti realization of the motivic cohomology spectrum. Following a similar strategy as in \cite{Levine:comparison} in the nonequivariant case, we show that the equivariant Betti realization takes the motivic cohomology spectrum $\HH\Z$ to the Bredon cohomology spectrum $\HH\ul{\Z}$. We then reinterpret the Beilinson-Lichtenbaum conjectures and establish an equivariant version of Suslin-Voevodsky's theorem \cite{SV:singular} on Suslin homology.

\begin{lemma}
 For any $X$ in $\Sch/k$, the natural map 
 $\L\RRe^{C_{2}}_{B,\phi}(\Sigma^{\infty}_{\P^{1}}X_+) \to
 \Sigma^{\infty}_{S^{1+\sigma}}X(\C)_{+}^{an}$ is an isomorphism in $\SH_{C_2}$.
\end{lemma}
\begin{proof}
 Since $k$ admits resolution of singularities there is a proper $cdh$ hypercover $X_{\bullet}\to X$ such that each $X_{n}$ is a smooth $k$-scheme. It follows from \cite{Voevodsky:cdh} that 
 $|\Sigma^{\infty}_{\P^{1}}X_{\bullet\,+}|\to \Sigma^{\infty}_{\P^{1}}X_+$ is a stable equivalence in $\SH_{k}$. Each $X_{n\,+}$ is cofibrant. It follows that we have a natural isomorphism $\L\RRe^{C_{2}}_{B,\phi}\Sigma^{\infty}_{\P^1}X_+ \iso 
 |\Sigma^{\infty}_{S^{1+\sigma}}X(\C)_{\bullet\,+}|$ in $\SH_{C_2}$. 
 
To see that $|\Sigma^{\infty}_{S^{1+\sigma}}X(\C)_{\bullet\,+}^{an}| \to \Sigma^{\infty}_{S^{1+\sigma}}X(\C)_{+}^{an}$ is an isomorphism in $\SH_{C_2}$ it suffices to check that this map induces an isomorphism in $\SH$ after applying the geometric fixed points functors $\Phi^{C_2}$ and $\Phi^{e}$. Recall that in general we have that the geometric fixed points of a suspension spectrum is given by the suspension spectrum of the fixed points: $\Phi^{H} \Sigma_{S^{G}}^{\infty} Y = \Sigma^{\infty} Y^{H}$. Therefore we have that the $C_2$-geometric fixed points of the above map is
$|\Sigma^{\infty}X(\R)_{\bullet\,+}^{an}| \to \Sigma^{\infty}X(\R)_{+}^{an}$. If $W\to Y$ is a proper $cdh$-cover of real varieties then $W(\R)^{an}\to Y(\R)^{an}$ is a surjective
proper map. In particular, it is a map of universal cohomological descent \cite[5.3.5]{Deligne:hodgeIII}. It follows that 
$H^{*}(|X(\R)_{\bullet\,+}^{an}|,A) \to H^{*}(X(\R)_{+}^{an},A)$ is an isomorphism for all abelian groups $A$. In particular, 
$|X(\R)_{\bullet\,+}^{an}| \to X(\R)_{+}^{an}$ induces a stable equivalence on suspension spectra. A similar analysis for the $e$-geometric fixed points shows that 
$|\Sigma^{\infty}X(\C)_{\bullet\,+}^{an}| \to \Sigma^{\infty}X(\C)_{+}^{an}$
is a stable equivalence as well.
 \end{proof}

\begin{lemma}
 The natural map $$
 \L\RRe^{C_{2}}_{B,\phi}(\Sigma^{\infty}_{\P^1}\Sym^{N}(\Sigma^{m}_{\P^{1}}Y_+))
 \to \Sigma^{\infty}_{S^{1+\sigma}}\Sym^{N}(\Sigma^{m}_{S^{1+\sigma}}Y(\C)^{an}_{+})
 $$ 
 is an isomorphism in $\SH_{C_2}$ for any $N$, $m$ and any $Y$ in $\sm/k$.
\end{lemma}
\begin{proof}
 The argument is identical to \cite[Lemma 5.4]{Levine:comparison}. The key point is that there is a homotopy pushout square in $\MS(k)$ of the form
 $$
 \xymatrix{
 \Sym^{N}(X, A) \ar[r]\ar[d] & 
 \Sym^{N}(X) \ar[d] \\
 \Sym^{N-1}(\Sigma^m_{\P^1}Y_+) \ar[r] & \Sym^{N}(\Sigma^m_{\P^1}Y_+)
 }
 $$
 where $X= (\P^{1})^m\times Y_+$ and $A$ is the closed subscheme of points 
 $(x_1,\ldots, x_m, y)$ such that some $x_i = \infty$.  The previous lemma applied to the top two vertices and induction on $N$ applied to the lower left vertex yields the result.
\end{proof}

 As in \cite{Levine:comparison} we write  
$$
 (\Sigma^{\infty}_{\P^{1}}X_+)^{tr}_{eff}:= (\Sym^{\infty}X_+, \Sym^{\infty}(\Sigma_{\P^{1}}X_{+}), \Sym^{\infty}(\Sigma_{\P^1}^2X_+),\ldots)
$$
and together with the obvious structure maps. 
Similarly for a $C_2$-space $W$ we have the $C_2$-spectrum $(\Sigma^{\infty}_{S^{1+\sigma}}W_+)^{tr}_{eff}:= 
\{ \Sym^{\infty}(\Sigma_{S^{1+\sigma}}^mW_+)\}_{m\geq 0}$, equipped with the obvious structure maps. 

\begin{proposition}
For any smooth $X$ there is a natural isomorphism in $\SH_{C_2}$ 
 $$
 \L\RRe^{C_2}_{B,\phi} (\Sigma^{\infty}_{\P^{1}}X_+)^{tr}_{eff} \iso (\Sigma^{\infty}_{S^{1+\sigma}}X(\C)^{an}_+)^{tr}_{eff}.
 $$
\end{proposition}
\begin{proof}
 We have the natural isomorphism 
 $\colim_{n}(\Sigma^{\infty}_{\P^{1}}E_{n})[n]\iso E$ in $\SH_{k}$, where $D[n]$ is the shifted spectrum given by $(D[n])_{i} = D_{i-n}$. Similarly we have the natural isomorphism  
 $\colim_{n}(\Sigma^{\infty}_{S^{1+\sigma}}F_{n})[n]\iso F$ in $\SH_{C_2}$. Since $\L\RRe^{C_2}_{B,\phi}$ preserves homotopy colimits and shifts, the result follows from the previous lemma.
\end{proof}

\begin{theorem}\label{thm:bettibredon}
 Let $\Lambda$ be an abelian group. There is an isomorphism in $\SH_{C_2}$
 $$
 \L\RRe^{C_2}_{B,\phi}(\HH \Lambda) \iso \HH\ul{\Lambda}.
 $$
\end{theorem}
\begin{proof}
Since $\HH \Lambda = \HH\Z \wedge \MM \Lambda$ and $\HH\ul{\Lambda} = \HH\ul{\Z}\wedge \MM \Lambda$, where 
$\MM \Lambda$ is a Moore spectrum for $\Lambda$, and $\L\RRe^{C_2}_{B,\phi}(\MM \Lambda) = \MM \Lambda$, it suffices to establish the result for $A=\Z$.
 The motivic cohomology spectrum $\HH\Z$ is given by 
 $\HH\Z_{n} = \Z^{tr}((\P^1)^{\wedge n})$ and equipped with the obvious structure maps. The natural map $(\SS_{k})^{tr}_{eff}\to \HH\Z$ is an isomorphism in $\SH_{k}$ by \cite[Lemma 5.9]{Levine:comparison}.
 It follows from \cite[Proposition 3.7]{DS:equiDT} that the spectrum
 $\{\Z S^{n(1+\sigma)}\}_{n\geq 0}$ is a model for $\HH\ul{\Z}$, \emph{i.e.}~ it represents Bredon cohomology with coefficients in the constant Mackey functor $\ul{\Z}$. 
 It follows from 
 \cite[Corollary A.7]{Dugger:krss} that the natural map $(\SS_{C_2})^{tr}_{eff} \to \{\Z S^{n(1+\sigma)}\}_{n\geq 0}$ is an equivariant weak equivalence. By the previous proposition, $\L\RRe^{C_{2}}_{B,\phi}((\SS_{k})^{tr}_{eff}) = (\SS_{C_2})^{tr}_{eff}$ and the result follows.
 \end{proof}

The Beilinson-Lichtenbaum conjectures assert that for any smooth variety $X$ over a field $k$, any $n>1$, and $q\geq 0$,  the generalized cycle map
$$
H_{\mcal{M}}^{p+q\alpha}(X,\Z/n) \to 
H^{p+q}_{\et}(X, \mu_{n}^{\otimes q})
$$
is an isomorphism for $p\leq 0$ and an injection for $p=1$. By a theorem of Suslin-Voevodsky \cite{SV:BK}, these conjectures are equivalent to the Bloch-Kato conjectures. In turn, these have been resolved by Voevodsky in case $n=2^{\ell}$ and in general by Voevodsky and Rost. Suppose now that $k=\R$. The \'etale cohomology (with finite coefficients) of the real variety $X$ can be identified with the Borel cohomology of $X(\C)$. On the other hand $\RRe^{C_{2}}_{B}$ induces a comparison map between motivic cohomology and Bredon cohomology and we would like to reinterpret the Beilinson-Lichtenbaum conjectures as a statement concerning this comparison. 
When $2$ is invertible in the coefficient group this is straightforward.
In \cite{HV:VT} the first author and M. Voineagu treat the case of coefficient group $\Z/2^{\ell}$ by carefully comparing various cycle maps together with a computation that Bredon and Borel cohomology agree in the appropriate range. This reinterpretation of Voevodsky's theorem applies more generally to the Betti realization for an embedding of a real closed field into $\R$. 

\begin{theorem}\label{thm:BL}
 Let $\phi:k\hookrightarrow\R$ be an embedding with $k$ real closed  and $X$ a smooth $k$-variety. For any $n\geq 1$, and any 
 $q\geq 0$ the map
 $$
 H^{s+q\alpha}_{\mcal{M}}(X,\Z/n) \to H^{s+q\sigma}(X(\C),\ul{\Z/n}),
 $$
 induced by $\RRe^{C_{2}}_{B,\phi}$, is an isomorphism for $s\leq 0$ and an injection for $s=1$.
\end{theorem}
\begin{proof}
 Motivic cohomology forms a pretheory with transfers. Applying  \cite[Theorem 1]{RO:rigidity},\footnote{This rigidity result is stated for dense subfields of a henselian valued field. Unfortunately $\R$ can't be equipped with a nontrivial henselian valuation.
 However, the proof of their result relies only on the density lemma 
 \cite[Lemma 1]{RO:rigidity} which is valid for a real closed subfield of $\R$, with the classical topology. This is well-known, see e.g.~ \cite[Lemma 4]{KatoSaito} for a proof.} 
 we have that the base change $\phi^*:H^{s+q\alpha}_{\mcal{M}}(X,\Z/n)\to H^{s+q\alpha}_{\mcal{M}}(X_{\R},\Z/n)$ is an isomorphism so it suffices to treat the case $k=\R$.

 Suppose that $2$ is invertible in $\Z/n$ and write $p:\spec(\C)\to \spec(\R)$ for the canonical map. Using \aref{prop:bettichange} we have the commutative diagram induced by $X_{\C}\to X$
 $$
 \xymatrix{
 H^{s+q\alpha}_{\mcal{M}}(X,\Z/n) \ar[r]^{p^*}\ar[d]_{\RRe^{C_2}_{B}} & H^{s+q\alpha}_{\mcal{M}}(X_{\C},\Z/n) \ar[r]^{p_\#}\ar[d]^{\RRe_{B}} &
 H^{s+q\alpha}_{\mcal{M}}(X,\Z/n) \ar[d]^{\RRe_{B}^{C_2}} \\
  H^{s+q\sigma}(X(\C),\ul{\Z/n}) \ar[r] &
   H^{s+q}_{\mathrm{sing}}(X(\C),\Z/n)\ar[r] &
    H^{s+q\sigma}(X(\C),\ul{\Z/n}).
 }
 $$
 The middle arrow is an isomorphism for $s\leq 0$ and an injection for $s=1$.
 The horizontal maps are multiplication by $2$, hence isomorphisms. The result thus follows for coefficient groups in which $2$ is invertible. 
  
 It remains to treat the case $\Z/2^{\ell}$.
 The cycle map $H^{s+q\alpha}_{\mcal{M}}(X,\Z/2^{\ell}) \to H^{s+q\sigma}(X(\C),\ul{\Z/2^{\ell}})$
 considered in \cite{HV:VT} 
 is  induced by the map of simplicial abelian groups (for $q\geq 0$)
 $$
\frac{ \Hom_{\R}(X\times \Delta^{\bullet}_{\R}, \Sym^{\infty}\P^{n})^{+}}{\Hom_{\R}(X\times \Delta^{\bullet}_{\R}, \Sym^{\infty}\P^{n-1})^{+}} \to \Hom_{C_{2}\btop}((X(\C)\times \Delta^{\bullet}_{top})_+, \Z (S^{n(1+\sigma)}))
 $$  
 obtained by sending an algebraic map of real varieties to its associated equivariant continuous map of $C_{2}$-spaces. This agrees with the map considered here. By 
 \cite[Theorem 1.5, Proposition 5.1]{HV:VT} it induces an isomorphism for $s\leq 0$ and an injection for $s=1$.
  \end{proof}

 We finish with an equivariant version of Suslin-Voevodsky's theorem \cite{SV:singular} that over an algebraically closed field Suslin homology agrees with \'etale homology. To set the stage, fix a real embedding $\phi:k\hookrightarrow \R$ and consider the subcategory  of motivic spectra  $X$ such that 
 $ \L\RRe^{C_{2}}_{B,\phi}$ induces an isomorphism $[S^{n},X]_{k} \iso
 [S^{n}, \L\RRe^{C_{2}}_{B,\phi}(X)]_{C_{2}}$ 
for all $n$. This is a localizing subcategory of $\SH_{k}$ and we show that it contains all effective torsion motives. 
 If the motivic slice tower were convergent we would be able to show more 
 generally that it contains all effective torsion motivic spectra (i.e.~ the localizing subcategory generated by 
 $\Sigma^{s}_{S^1}\Sigma^{t}_{\P^{1}}\Sigma^{\infty}_{\P^{1}}X/N$ for any 
 $s\in \Z$, $t\geq 0$, $N>1$, and smooth $X$). 
 \begin{theorem}\label{thm:SV}
Let $k$ be a real closed field and $\phi:k\hookrightarrow\R$ be an embedding. Let $E$ be in the smallest localizing subcategory of $\SH_{k}$ containing $X_{+}\wedge \HH\Z/r$ for  any smooth projective $X$ and $r>1$. 
Then for any $n$, the equivariant Betti realization induces an isomorphism
\begin{equation*}
\RRe^{C_{2}}_{B,\phi}:[S^{n}, E ]_{k} \xrightarrow{\iso} 
[S^{n}, \RRe^{C_{2}}_{B,\phi}(E)]_{C_2}.
\end{equation*}
 \end{theorem}
\begin{proof}
It suffices to show that $[S^{n}, X\wedge \HH \Z/r ]_{k} \to 
[S^{n}, X_{\R}(\C)\wedge \HH\ul{\Z/r}]_{C_2}$ is an isomorphism for any smooth projective $X$.
As in the previous theorem, using \cite[Theorem 1]{RO:rigidity}, we are reduced to the case $k=\R$. 
Tracing through definitions, it suffices to show that the map
 \begin{equation*}
\Z^{tr}(X)(\Delta^{\bullet}_{\R})\otimes\Z/r=\Hom_{\R}(\Delta^{\bullet}_{\R}, \Sym^{\infty}X)^{+}\otimes\Z/r
\to \Hom_{C_{2}\btop}(\Delta^{\bullet}_{top}, \Z X(\C))\otimes \Z/r
 \end{equation*}
of simplicial abelian groups, obtained by sending an algebraic map of real varieties to its associated equivariant continuous map of $C_{2}$-spaces, is a homotopy equivalence. Note that this last simplicial abelian group equals
$\sing_{\bullet}(\Z X(\C))^{C_{2}}\otimes\Z/r$.

That this map is a homotopy equivalence can be deduced by a variant of some arguments of Friedlander-Walker \cite{FW:ratisos} as follows.
First, for a presheaf $F$ on $\sch/\R$,  define $F(\Delta^{d}_{top}) = \colim_{\Delta^{d}_{top}\to W(\R)}F(W)$ where the colimit ranges over continuous maps and $W$ a finite type real variety. Note that if $F$ is the presheaf represented by a real variety $Y$ then $F(\Delta^{\bullet}_{top}) = \sing_{\bullet}Y(\R)$. 
Consider the presheaf of simplicial abelian groups
$$   
G(-):=  
\Z^{tr}(X)(-\times\Delta^{\bullet}_{\R})\otimes \Z/r.
$$  
Note that $G(\Delta^{\bullet}_{top})$ and $[\Hom_{C_{2}\btop}(\Delta^{\bullet}_{top}, \Sym^{\infty}X(\C))]^{+}\otimes\Z/r$ are naturally homotopic. 
Combined with Quillen's theorem \cite[Appendix Q]{FM:filtrations} on homotopy group completions of simplicial abelian monoids and the fact that $(\Z X(\C))^{C_2}$ is the homotopy group completion of 
$(\mathbb{N} X(\C))^{C_{2}}$,  we find that there is a natural homotopy equivalence of simplicial abelian groups 
$G(\Delta^{\bullet}_{top}) \wkeq \sing_{\bullet}(\Z X(\C))^{C_2}\otimes \Z/r$. It thus suffices to show that the map 
$G(\R)\to G(\Delta^{\bullet}_{top})$ (induced by the projections $\Delta^{d}_{top}\to \ast$) is a homotopy equivalence. This is easily seen via the same argument as in  \cite[Proposition 5.1]{HV:VT}.
 \end{proof}

\bibliographystyle{plain}
\bibliography{E2M}

\end{document}